\documentclass[11pt]{article}

\usepackage{geometry}
\usepackage{amsmath}
\usepackage{amssymb}
\usepackage{amsthm}
\usepackage{mathtools}
\usepackage{mathrsfs}  

\usepackage{graphicx}
\usepackage{tikz}
\usepackage{tikz-cd}   
\usetikzlibrary{arrows.meta, positioning, decorations.markings}

\usepackage{microtype}
\usepackage[T1]{fontenc}
\usepackage{lmodern}

\usepackage[numbers,sort&compress]{natbib}

\usepackage[colorlinks=true,
            linkcolor=blue!70!black,
            citecolor=blue!70!black,
            urlcolor=blue!70!black]{hyperref}

\usepackage{doi}
\usepackage[capitalise, nameinlink]{cleveref}

\usepackage{enumitem}
\usepackage{xcolor}
\usepackage{booktabs}  
\usepackage{appendix}  
\usepackage{caption}

\newtheoremstyle{ghriststyle}
    {12pt}                    
    {12pt}                    
    {\itshape}                
    {}                        
    {\scshape}                
    {.}                       
    {.5em}                    
    {}                        

\newtheoremstyle{ghristdefstyle}
    {12pt}                    
    {12pt}                    
    {\normalfont}             
    {}                        
    {\scshape}                
    {.}                       
    {.5em}                    
    {}                        

\newtheoremstyle{ghristremarkstyle}
    {8pt}                     
    {8pt}                     
    {\normalfont}             
    {}                        
    {\itshape}                
    {.}                       
    {.5em}                    
    {}                        

\theoremstyle{ghriststyle}
\newtheorem{theorem}{Theorem}[section]
\newtheorem{lemma}[theorem]{Lemma}
\newtheorem{proposition}[theorem]{Proposition}
\newtheorem{corollary}[theorem]{Corollary}

\theoremstyle{ghristdefstyle}
\newtheorem{definition}[theorem]{Definition}
\newtheorem{example}[theorem]{Example}

\theoremstyle{ghristremarkstyle}
\newtheorem{remark}[theorem]{Remark}

\makeatletter

\newcommand{\thanksblock}[1]{\gdef\@thanksblock{#1}}
\gdef\@thanksblock{} 

\renewcommand{\maketitle}{%
  \begin{center}%
    {\LARGE\bfseries \@title \par}%
    \vskip 1em%
    {\large \lineskip .5em%
      \begin{tabular}[t]{c}%
        \@author
      \end{tabular}\par}%
    \ifx\@thanksblock\@empty\else
      \vskip .75em%
      {\small \@thanksblock\par}%
    \fi
  \end{center}%
  \vskip 1.5em%
}

\makeatother

\date{} 

\usepackage{titlesec}

\titleformat{\section}
    {\normalfont\large\bfseries}
    {\thesection.}
    {0.5em}
    {}

\titleformat{\subsection}
    {\normalfont\normalsize\bfseries}
    {\thesubsection.}
    {0.5em}
    {}

\titleformat{\subsubsection}
    {\normalfont\normalsize\itshape}
    {\thesubsubsection.}
    {0.5em}
    {}

\renewenvironment{abstract}{%
    \begin{center}
        \textsc{Abstract}
    \end{center}
    \begin{quotation}
    \noindent
    \small
}{%
    \end{quotation}
    \vspace{1em}
    \normalsize
}

\title{Selective Adaptation of Beliefs and Communication on Cellular Sheaves}

\author{Vicente Bosca \and Robert Ghrist}

\thanksblock{%
Department of Mathematics, University of Pennsylvania\\
\texttt{vicenteg@sas.upenn.edu}\quad
\texttt{ghrist@math.upenn.edu}%
}

\begin{document}
\maketitle


\begin{abstract}
We extend opinion dynamics on discourse sheaves to incorporate {\em directional stubbornness}: 
agents may hold fixed positions in specified directions of their opinion stalk while remaining 
flexible in others. This converts the equilibrium problem from harmonic extension to a forced 
sheaf equation: the free-opinion component satisfies a sheaf Poisson equation with forcing 
induced by the clamped directions.

We develop a parallel theory for {\em selective learning} of expression policies. When only a 
designated subset of incidence maps may adapt, the resulting gradient flow is sheaf diffusion 
on an auxiliary structure sheaf whose global sections correspond to sheaf structures making a 
fixed opinion profile publicly consistent.

For joint evolution of beliefs and expressions, we give conditions (and regularized variants) 
guaranteeing convergence to nondegenerate equilibria, excluding spurious agreement via vanishing 
opinions or trivialized communication maps. Finally, we derive stagnation bounds in terms of the 
rate ratio between opinion updating and structural adaptation, quantifying when rapid rhetorical 
accommodation masks nearly unchanged beliefs, and conversely when flexible beliefs conform to 
rigid communication norms.
\end{abstract}


\section{Introduction}\label{sec:introduction}

Disagreement is endemic to social life, yet societies function. How is conflict managed? Sometimes people change their minds. But often they do not: they change how they \emph{talk} instead. A politician holds firm convictions while tailoring rhetoric to each audience. Colleagues who disagree on substance adopt diplomatic language to preserve working relationships. Families develop implicit rules about which topics permit frank discussion and which require careful navigation. The distinction between \emph{what one believes} and \emph{how one expresses it} is fundamental to human discourse~\cite{kuran1997,kelman1958,ye2019expressedprivate}.

Mathematical models of opinion dynamics have largely ignored this distinction. 
Classical averaging models~\cite{abelson1964,degroot1974,french1956,friedkin1990} treat communication as transparent: agents exchange opinions directly, and network topology alone determines influence.
Nonlinear alternatives such as bounded-confidence models capture selective interaction based on opinion proximity~\cite{deffuant2000,hegselmann2002,dittmer2001,blondel2010,canuto2012}.
The framework of \emph{discourse sheaves} introduced by Hansen and Ghrist~\cite{hansen2020} enriches this picture by allowing multidimensional opinions and heterogeneous expression channels. Their ``learning to lie'' dynamics permits expression maps to evolve while opinions stay fixed, and they describe joint evolution of both. They also introduce stubborn agents whose opinions remain fixed while influencing their neighbors~\cite{taylor1968,ghaderi2014,tian2018}.
However, their analysis treats stubbornness as an all-or-nothing property of vertices, rather than a directional constraint: an agent is either entirely stubborn or entirely flexible.

This paper extends the theory of stubborn agents to allow \emph{selective rigidity}. An agent may hold fixed positions in certain directions of their opinion space while remaining flexible in others, and may maintain rigid expression policies toward some neighbors while adapting toward others. A vertex may be stubborn in belief but diplomatic in expression, or vice versa. The resulting framework captures the full combinatorics of partial rigidity in networked discourse.

When opinions and expressions evolve simultaneously, we study four scenarios distinguished by which communication channels are permitted to adapt. We prove convergence and characterize when perfect agreement is achievable. Conservation laws ensure the system cannot achieve spurious agreement through vanishing opinions or ceasing communication; equilibrium reflects genuine accommodation rather than disengagement. Conditions on the ratio of update rates then reveal when one variable effectively stagnates while the other equilibrates. When expressions adapt sufficiently faster than beliefs, surface harmony masks unchanged private convictions. The reverse regime produces flexible individuals conforming to rigid communication norms.

The paper proceeds as follows. Section~\ref{sec:prelim} recalls discourse sheaves and the Hansen--Ghrist framework. Section~\ref{sec:stubborn-opinions} introduces stubborn opinions, develops the sheaf of free opinions, and shows that equilibrium is characterized by a sheaf Poisson equation on the sheaf of free opinions. Section~\ref{sec:stubborn-expressions} treats stubborn expressions in parallel, constructing the sheaf of admissible structures and establishing a duality between constrained opinion dynamics and constrained structure adaptation. Section~\ref{sec:joint} analyzes joint dynamics under the four constraint scenarios, providing convergence guarantees and conservation laws that prevent collapse to trivial equilibria. Section~\ref{sec:timescale} develops the timescale separation analysis and derives the stagnation bounds.


\section{Preliminaries: Discourse Sheaves and Opinion Dynamics}
\label{sec:prelim}

We briefly recall the framework of cellular sheaves for opinion dynamics introduced by Hansen and Ghrist~\cite{hansen2020}. The key insight of their approach is that social networks carry richer structure than a simple graph: each agent holds private opinions, each relationship involves a shared communication space, and agents may express their opinions differently to different neighbors. Cellular sheaves capture this layered structure in a single mathematical object.

\subsection{Cellular Sheaves on Graphs}\label{ssec:cellular-sheaves}

Let $G = (V, E)$ be a finite connected graph.

\begin{definition}\label{def:cellular-sheaf}
A \emph{cellular sheaf} $\mathcal{F}$ on $G$ assigns:
\begin{itemize}
    \item a finite-dimensional real inner product space $\mathcal{F}(v)$, called the \emph{stalk}, to each vertex $v \in V$,
    \item a finite-dimensional real inner product space $\mathcal{F}(e)$ to each edge $e \in E$,
    \item a linear \emph{restriction map} $\mathcal{F}_{v \unlhd e}: \mathcal{F}(v) \to \mathcal{F}(e)$ for each incidence $v \in e$.
\end{itemize}
In the context of opinion dynamics, we call such a sheaf a \emph{discourse sheaf}.
\end{definition}

The vertex stalks $\mathcal{F}(v)$ represent the private opinion spaces of agents, the edge stalks $\mathcal{F}(e)$ represent the discourse spaces where neighboring agents communicate, and the restriction maps encode how agents express their private opinions in public discourse. An agent may express different aspects of their opinion to different neighbors, so each edge incident to $v$ has its own restriction map.

\subsection{Cochains and the Sheaf Laplacian}\label{ssec:cochains-laplacian}

The space of \emph{$0$-cochains} is
\begin{align*}
C^0(G; \mathcal{F}) = \bigoplus_{v \in V} \mathcal{F}(v),
\end{align*}
equipped with the inner product $\langle x, y \rangle = \sum_{v \in V} \langle x_v, y_v \rangle_{\mathcal{F}(v)}$. A $0$-cochain $x \in C^0(G; \mathcal{F})$ assigns to each vertex $v$ a vector $x_v \in \mathcal{F}(v)$, representing the private opinions of agent $v$. Similarly, the space of \emph{$1$-cochains} is
\begin{align*}
C^1(G; \mathcal{F}) = \bigoplus_{e \in E} \mathcal{F}(e),
\end{align*}
equipped with the analogous inner product. A $1$-cochain assigns a vector in $\mathcal{F}(e)$ to each edge.

Choosing an arbitrary orientation on each edge $e = u \to v$, the \emph{coboundary operator} $\delta: C^0(G; \mathcal{F}) \to C^1(G; \mathcal{F})$ is defined by
\begin{align*}
(\delta x)_e = \mathcal{F}_{v \unlhd e}(x_v) - \mathcal{F}_{u \unlhd e}(x_u).
\end{align*}
The value $(\delta x)_e$ measures the discrepancy between the expressed opinions of the endpoints of edge $e$: it is zero precisely when both agents, expressing their private opinions through their respective restriction maps, arrive at the same point in the shared discourse space.

The \emph{sheaf Laplacian} is the positive semidefinite operator 
\begin{align}
L_{\mathcal{F}} = \delta^T \delta: C^0(G; \mathcal{F}) \to C^0(G; \mathcal{F}),
\end{align}
where the transpose is taken with respect to the cochain inner products. The quadratic form $\langle x, L_{\mathcal{F}} x \rangle = \|\delta x\|^2$ measures the total disagreement in the network.

\subsection{Cohomology and Global Sections}\label{ssec:cohomology}

The \emph{zeroth cohomology}, or space of \emph{global sections}, is
\begin{align*}
H^0(G; \mathcal{F}) = \ker(\delta) = \ker(L_{\mathcal{F}}).
\end{align*}
Global sections are opinion distributions in which all expressed opinions agree: every pair of adjacent agents, when expressing their private opinions through the restriction maps, arrive at the same point in the shared discourse space. The equality $\ker(\delta) = \ker(L_{\mathcal{F}})$ is known as the sheaf Hodge theorem~\cite[Theorem 2.2]{hansen2020}.

\subsection{Opinion Dynamics on Discourse Sheaves}\label{ssec:opinion-dynamics}

The fundamental convergence result for opinion dynamics on discourse sheaves is the following:

\begin{theorem}[Hansen and Ghrist {\cite[Theorem 4.1]{hansen2020}}]
\label{thm:hansen-ghrist}
Solutions to the sheaf diffusion equation
\begin{equation}
\frac{dx}{dt} = -\alpha L_{\mathcal{F}} x, \qquad \alpha > 0,
\end{equation}
converge exponentially to the orthogonal projection of $x(0)$ onto $H^0(G; \mathcal{F})$.
\end{theorem}

Under this dynamics, agents adjust their private opinions to minimize total disagreement. The limit represents a harmonic configuration: a compromise that balances the conflicting demands imposed by the network structure.

When a subset $U \subseteq V$ of agents maintains fixed opinions, the relevant object becomes the \emph{relative cohomology} $H^0(G, U; \mathcal{F})$, consisting of global sections that vanish on $U$. In this setting, the dynamics reduce to a Dirichlet problem on the complement of $U$, and the limit is the harmonic extension of the boundary values imposed by stubborn agents that is nearest to the initial state. Uniqueness of this limit is guaranteed when $H^0(G, U; \mathcal{F}) = 0$.


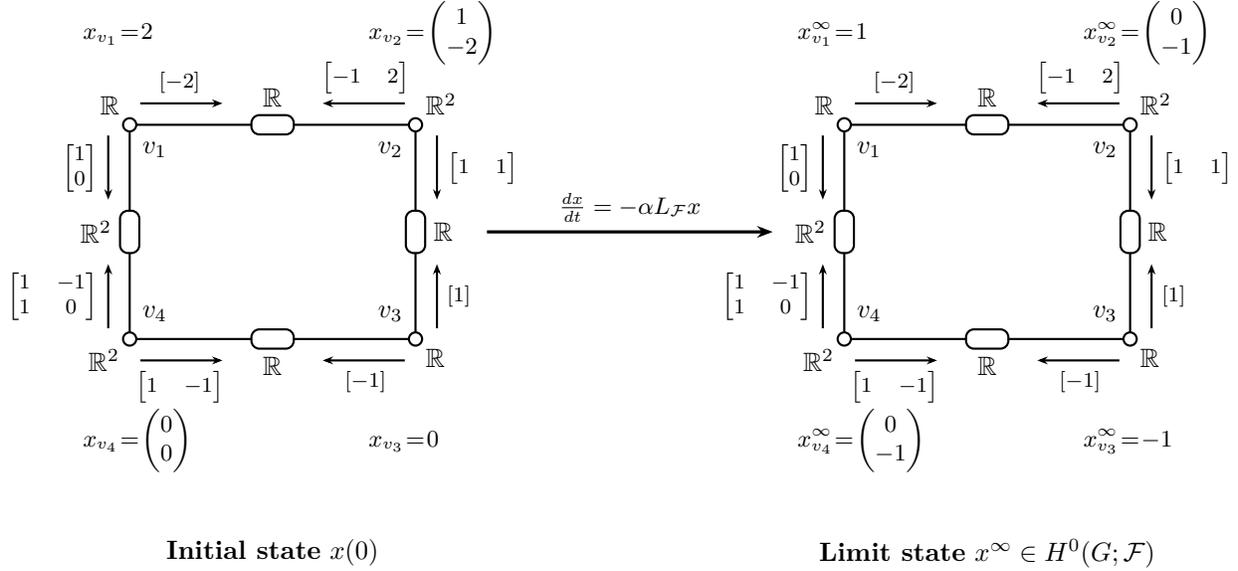
\begin{figure}[ht]
\centering
\begin{tikzpicture}[
    vertex/.style={circle, fill=white, draw=black, thick, minimum size=5pt, inner sep=0pt},
    edge stalk h/.style={rectangle, rounded corners=3pt, fill=white, draw=black, thick, minimum width=16pt, minimum height=4pt},
    edge stalk v/.style={rectangle, rounded corners=3pt, fill=white, draw=black, thick, minimum width=4pt, minimum height=16pt},
    arrow/.style={-{Stealth[length=4pt]}, thick},
    every node/.style={font=\small},
    scale=0.95
]


\draw[thick] (0,0) -- (4.0,0) -- (4.0,3) -- (0,3) -- cycle;

\node[edge stalk h] (e_top_L) at (2.0,3) {};
\node[edge stalk h] (e_bot_L) at (2.0,0) {};
\node[edge stalk v] (e_left_L) at (0,1.5) {};
\node[edge stalk v] (e_right_L) at (4.0,1.5) {};

\node[vertex] (v1_L) at (0,3) {};
\node[vertex] (v2_L) at (4.0,3) {};
\node[vertex] (v3_L) at (4.0,0) {};
\node[vertex] (v4_L) at (0,0) {};

\node at (0.35,2.65) {$v_1$};
\node at (3.65,2.65) {$v_2$};
\node at (3.65,0.35) {$v_3$};
\node at (0.35,0.35) {$v_4$};

\node[above left] at (v1_L) {$\mathbb{R}$};
\node[above right] at (v2_L) {$\mathbb{R}^2$};
\node[below right] at (v3_L) {$\mathbb{R}$};
\node[below left] at (v4_L) {$\mathbb{R}^2$};

\node[anchor=base east, font=\footnotesize, inner sep=0pt] at (-0.15,4.2) {$x_{v_1}$};
\node[anchor=base, font=\footnotesize, inner sep=0pt] at (0,4.2) {$=$};
\node[anchor=base west, font=\footnotesize, inner sep=0pt] at (0.15,4.2) {$2$};
\node[anchor=base east, font=\footnotesize, inner sep=0pt] at (3.85,4.2) {$x_{v_2}$};
\node[anchor=base, font=\footnotesize, inner sep=0pt] at (4.0,4.2) {$=$};
\node[anchor=base west, font=\footnotesize, inner sep=0pt] at (4.15,4.2) {$\begin{pmatrix} 1 \\ -2 \end{pmatrix}$};
\node[anchor=base east, font=\footnotesize, inner sep=0pt] at (-0.15,-1.5) {$x_{v_4}$};
\node[anchor=base, font=\footnotesize, inner sep=0pt] at (0,-1.5) {$=$};
\node[anchor=base west, font=\footnotesize, inner sep=0pt] at (0.15,-1.5) {$\begin{pmatrix} 0 \\ 0 \end{pmatrix}$};
\node[anchor=base east, font=\footnotesize, inner sep=0pt] at (3.85,-1.5) {$x_{v_3}$};
\node[anchor=base, font=\footnotesize, inner sep=0pt] at (4.0,-1.5) {$=$};
\node[anchor=base west, font=\footnotesize, inner sep=0pt] at (4.15,-1.5) {$0$};

\node[above=3pt] at (e_top_L) {$\mathbb{R}$};
\node[below=3pt] at (e_bot_L) {$\mathbb{R}$};
\node[left=3pt] at (e_left_L) {$\mathbb{R}^2$};
\node[right=3pt] at (e_right_L) {$\mathbb{R}$};

\draw[arrow] (0.15,3.3) -- (1.3,3.3);
\node[above, font=\scriptsize] at (0.7,3.3) {$[-2]$};

\draw[arrow] (-0.3,2.85) -- (-0.3,1.95);
\node[left, font=\scriptsize] at (-0.3,2.4) {$\begin{bmatrix} 1 \\ 0 \end{bmatrix}$};

\draw[arrow] (3.85,3.3) -- (2.7,3.3);
\node[above, font=\scriptsize] at (3.3,3.3) {$\begin{bmatrix} -1 & 2 \end{bmatrix}$};

\draw[arrow] (4.3,2.85) -- (4.3,1.95);
\node[right, font=\scriptsize] at (4.3,2.4) {$\begin{bmatrix} 1 & 1 \end{bmatrix}$};

\draw[arrow] (4.3,0.15) -- (4.3,1.05);
\node[right, font=\scriptsize] at (4.3,0.6) {$[1]$};

\draw[arrow] (3.85,-0.3) -- (2.7,-0.3);
\node[below, font=\scriptsize] at (3.3,-0.3) {$[-1]$};

\draw[arrow] (0.15,-0.3) -- (1.3,-0.3);
\node[below, font=\scriptsize] at (0.7,-0.3) {$\begin{bmatrix} 1 & -1 \end{bmatrix}$};

\draw[arrow] (-0.3,0.15) -- (-0.3,1.05);
\node[left, font=\scriptsize] at (-0.3,0.6) {$\begin{bmatrix} 1 & -1 \\ 1 & 0 \end{bmatrix}$};

\node[below=72pt] at (2.0,0) {\textbf{Initial state} $x(0)$};

\begin{scope}[xshift=10.0cm]

\draw[thick] (0,0) -- (4.0,0) -- (4.0,3) -- (0,3) -- cycle;

\node[edge stalk h] (e_top_R) at (2.0,3) {};
\node[edge stalk h] (e_bot_R) at (2.0,0) {};
\node[edge stalk v] (e_left_R) at (0,1.5) {};
\node[edge stalk v] (e_right_R) at (4.0,1.5) {};

\node[vertex] (v1_R) at (0,3) {};
\node[vertex] (v2_R) at (4.0,3) {};
\node[vertex] (v3_R) at (4.0,0) {};
\node[vertex] (v4_R) at (0,0) {};

\node at (0.35,2.65) {$v_1$};
\node at (3.65,2.65) {$v_2$};
\node at (3.65,0.35) {$v_3$};
\node at (0.35,0.35) {$v_4$};

\node[above left] at (v1_R) {$\mathbb{R}$};
\node[above right] at (v2_R) {$\mathbb{R}^2$};
\node[below right] at (v3_R) {$\mathbb{R}$};
\node[below left] at (v4_R) {$\mathbb{R}^2$};

\node[anchor=base east, font=\footnotesize, inner sep=0pt] at (-0.15,4.2) {$x^\infty_{v_1}$};
\node[anchor=base, font=\footnotesize, inner sep=0pt] at (0,4.2) {$=$};
\node[anchor=base west, font=\footnotesize, inner sep=0pt] at (0.15,4.2) {$1$};
\node[anchor=base east, font=\footnotesize, inner sep=0pt] at (3.85,4.2) {$x^\infty_{v_2}$};
\node[anchor=base, font=\footnotesize, inner sep=0pt] at (4.0,4.2) {$=$};
\node[anchor=base west, font=\footnotesize, inner sep=0pt] at (4.15,4.2) {$\begin{pmatrix} 0 \\ -1 \end{pmatrix}$};
\node[anchor=base east, font=\footnotesize, inner sep=0pt] at (-0.15,-1.5) {$x^\infty_{v_4}$};
\node[anchor=base, font=\footnotesize, inner sep=0pt] at (0,-1.5) {$=$};
\node[anchor=base west, font=\footnotesize, inner sep=0pt] at (0.15,-1.5) {$\begin{pmatrix} 0 \\ -1 \end{pmatrix}$};
\node[anchor=base east, font=\footnotesize, inner sep=0pt] at (3.85,-1.5) {$x^\infty_{v_3}$};
\node[anchor=base, font=\footnotesize, inner sep=0pt] at (4.0,-1.5) {$=$};
\node[anchor=base west, font=\footnotesize, inner sep=0pt] at (4.15,-1.5) {$-1$};

\node[above=3pt] at (e_top_R) {$\mathbb{R}$};
\node[below=3pt] at (e_bot_R) {$\mathbb{R}$};
\node[left=3pt] at (e_left_R) {$\mathbb{R}^2$};
\node[right=3pt] at (e_right_R) {$\mathbb{R}$};

\draw[arrow] (0.15,3.3) -- (1.3,3.3);
\node[above, font=\scriptsize] at (0.7,3.3) {$[-2]$};

\draw[arrow] (-0.3,2.85) -- (-0.3,1.95);
\node[left, font=\scriptsize] at (-0.3,2.4) {$\begin{bmatrix} 1 \\ 0 \end{bmatrix}$};

\draw[arrow] (3.85,3.3) -- (2.7,3.3);
\node[above, font=\scriptsize] at (3.3,3.3) {$\begin{bmatrix} -1 & 2 \end{bmatrix}$};

\draw[arrow] (4.3,2.85) -- (4.3,1.95);
\node[right, font=\scriptsize] at (4.3,2.4) {$\begin{bmatrix} 1 & 1 \end{bmatrix}$};

\draw[arrow] (4.3,0.15) -- (4.3,1.05);
\node[right, font=\scriptsize] at (4.3,0.6) {$[1]$};

\draw[arrow] (3.85,-0.3) -- (2.7,-0.3);
\node[below, font=\scriptsize] at (3.3,-0.3) {$[-1]$};

\draw[arrow] (0.15,-0.3) -- (1.3,-0.3);
\node[below, font=\scriptsize] at (0.7,-0.3) {$\begin{bmatrix} 1 & -1 \end{bmatrix}$};

\draw[arrow] (-0.3,0.15) -- (-0.3,1.05);
\node[left, font=\scriptsize] at (-0.3,0.6) {$\begin{bmatrix} 1 & -1 \\ 1 & 0 \end{bmatrix}$};

\node[below=72pt] at (2.0,0) {\textbf{Limit state} $x^\infty \in H^0(G; \mathcal{F})$};

\end{scope}

\draw[-{Stealth[length=6pt]}, very thick] (5.0,1.5) -- (9.0,1.5);
\node[above, font=\footnotesize] at (7.0,1.5) {$\frac{dx}{dt} = -\alpha L_{\mathcal{F}} x$};

\end{tikzpicture}
\caption{Sheaf diffusion on a discourse sheaf. \textbf{Left:} Initial opinion distribution $x(0)$ with discrepancy $\|\delta x(0)\|^2 = 6$. \textbf{Right:} The limit $x^\infty = P_{H^0}(x(0))$, the orthogonal projection onto the space of global sections. At equilibrium, all expressed opinions agree: both endpoints of each edge map to the same value in the discourse space ($-2$ on the top edge, $-1$ on the right, $1$ on the bottom, and $(1,0)^T$ on the left).}
\label{fig:sheaf-diffusion-example}
\end{figure}


\section{Stubborn Opinions}\label{sec:stubborn-opinions}

Not all opinions are equally negotiable: a person may be immovable on core values while remaining open to persuasion on peripheral issues. In the classical formulation of stubborn agents on discourse sheaves, certain vertices have their entire opinion spaces fixed~\cite{taylor1968,ghaderi2014,tian2018}. We generalize this to a setting where agents may be stubborn in specific conceptual directions while remaining flexible in others. This structure is captured by decomposing each agent's opinion space into fixed and free components.

\subsection{The Sheaf of Free Opinions}\label{ssec:sheaf-free-opinions}

To formalize partial stubbornness, we decompose each agent's opinion space into stubborn and free components, then construct a new sheaf that tracks only the free opinions.

\begin{definition}[Sheaf of Free Opinions]
Let $U \subseteq V$ be the set of stubborn vertices. For each $v \in V$, let $S_v \subseteq \mathcal{F}(v)$ be the subspace of stubborn directions, with $S_v = \{0\}$ for $v \notin U$. Let $T_v = S_v^\perp$ denote the orthogonal complement, the subspace of free opinions. The \emph{sheaf of free opinions} $\mathcal{Q}$ has stalks
\begin{align*}
\mathcal{Q}(v) = T_v, \qquad \mathcal{Q}(e) = \mathcal{F}(e),
\end{align*}
and restriction maps $\mathcal{Q}_{v \unlhd e}: T_v \to \mathcal{F}(e)$ defined by
\begin{align*}
\mathcal{Q}_{v \unlhd e} = \mathcal{F}_{v \unlhd e} \circ \iota_v,
\end{align*}
where $\iota_v: T_v \hookrightarrow \mathcal{F}(v)$ is the inclusion.
\end{definition}

Note that a vertex $v \in U$ may still have a nontrivial free subspace $T_v$; only its $S_v$-coordinates are clamped. Vertices in $F = V \setminus U$ have $S_v = \{0\}$, so their entire opinion vectors evolve. 

The sheaf $\mathcal{Q}$ inherits the discourse structure of $\mathcal{F}$ but restricts attention to the negotiable components of each agent's opinion. The edge stalks remain unchanged because the discourse spaces themselves are not constrained; only the opinions entering them are. The restriction maps of $\mathcal{Q}$ are simply the original maps composed with the inclusion of the free subspace.

The decomposition $\mathcal{F}(v) = S_v \oplus T_v$ induces a splitting of the global cochain space
\begin{align*}
C^0(G; \mathcal{F}) = C^0(S) \oplus C^0(\mathcal{Q}),
\end{align*}
where $C^0(S) = \bigoplus_v S_v$ collects all stubborn directions and $C^0(\mathcal{Q}) = \bigoplus_v T_v$ collects all free directions. We write $u \in C^0(S)$ for the vector encoding the stubborn opinion values, $\iota_S$ and $\iota_{\mathcal{Q}}$ for the natural inclusions into $C^0(G; \mathcal{F})$, and $P_S$ and $P_{\mathcal{Q}}$ for the corresponding orthogonal projections.

The sheaf $\mathcal{Q}$ is a subsheaf of $\mathcal{F}$: the vertex stalks satisfy $T_v \subseteq \mathcal{F}(v)$, the edge stalks coincide, and the restriction maps of $\mathcal{F}$ carry $T_v$ into $\mathcal{F}(e) = \mathcal{Q}(e)$. The quotient $\mathcal{S} = \mathcal{F}/\mathcal{Q}$ has stalks $\mathcal{S}(v) = S_v$ at vertices and $\mathcal{S}(e) = 0$ at edges (since $\mathcal{Q}(e) = \mathcal{F}(e)$), giving a short exact sequence of sheaves
\[
0 \to \mathcal{Q} \hookrightarrow \mathcal{F} \twoheadrightarrow \mathcal{S} \to 0.
\]
The induced long exact sequence in cohomology includes a connecting homomorphism $\partial: C^0(S) \to H^1(G; \mathcal{Q})$ that measures the obstruction to perfect agreement. When $\partial(u) = 0$, the stubborn configuration is \emph{compatible}: there exists an opinion distribution achieving zero disagreement across all edges. When $\partial(u) \neq 0$, the configuration is \emph{incompatible}: residual disagreement is unavoidable regardless of how free agents adjust. The cohomology class $\partial(u)$ quantifies this topological obstruction to harmony, with a torsor-theoretic interpretation developed in the context of affine sheaves by Gould~\cite{gould2024hilbert}; see Appendix~\ref{app:exact-sequences} for the full development.

\subsection{Constrained Opinion Dynamics}\label{ssec:constrained-opinion-dynamics}

With the sheaf of free opinions in hand, we can formulate the dynamics. Stubborn opinions remain fixed while free opinions evolve to reduce disagreement, subject to the influence of their stubborn neighbors.

To simplify notation, we write $\tilde{u} = \iota_S(u)$ and $\tilde{y} = \iota_{\mathcal{Q}}(y)$ for the embeddings into the full cochain space $C^0(G; \mathcal{F})$, so that $x = \tilde{u} + \tilde{y}$. Under the splitting $C^0(G; \mathcal{F}) = C^0(S) \oplus C^0(\mathcal{Q})$, the sheaf Laplacian decomposes as
\begin{align*}
L_{\mathcal{F}} = \begin{bmatrix} L_{SS} & L_{SQ} \\ L_{QS} & L_{QQ} \end{bmatrix},
\end{align*}
where $L_{QQ} = L_{\mathcal{Q}}$ is the Laplacian of the sheaf of free opinions. The off-diagonal block $L_{QS}$ captures how stubborn opinions influence the free dynamics: the term $L_{QS} u$ acts as a forcing that drives free opinions away from zero even when they would otherwise relax to consensus.

\begin{theorem}[Convergence to Equilibrium]
\label{thm:stubborn-directions}
Let $\mathcal{F}$ be a discourse sheaf on a connected graph $G = (V, E)$, and let $\mathcal{Q}$ be the sheaf of free opinions. Fix a stubborn configuration $u \in C^0(S)$, and let $x = \tilde{u} + \tilde{y}$ where $y \in C^0(\mathcal{Q})$ evolves according to the projected Laplacian dynamics
\begin{equation}
\label{eq:stubborn-dynamics}
\frac{dy}{dt} = -\alpha \bigl( L_{\mathcal{Q}} y + L_{QS} u \bigr), \qquad \alpha > 0.
\end{equation}

For any initial condition $y(0)$, the system converges exponentially to a limit $y^\infty \in C^0(\mathcal{Q})$ satisfying the sheaf Poisson equation

\begin{equation}\label{eq:poisson}
L_{\mathcal{Q}} y^\infty = -L_{QS} u.
\end{equation}
The solution set forms an affine subspace $\mathcal{A}$ with direction space $H^0(G; \mathcal{Q}) = \ker(L_{\mathcal{Q}})$, and the dynamics selects the point in $\mathcal{A}$ closest to the initial condition:
\begin{equation}
\label{eq:y_inf}
y^\infty = y(0) - L_{\mathcal{Q}}^\dagger P_{\mathcal{Q}} L_{\mathcal{F}} x(0).
\end{equation}
The total state $x^\infty = \tilde{u} + \tilde{y}^\infty$ minimizes the global disagreement energy subject to the stubborn constraints:
\begin{align*}
x^\infty \in \operatorname*{argmin}_{x:\, P_S(x) = u} \|\delta x\|^2.
\end{align*}
\end{theorem}

The proof proceeds by analyzing the block structure of the Laplacian. The key step is showing that the Poisson equation admits a solution, which requires $\operatorname{im}(L_{QS}) \subseteq \operatorname{im}(L_{\mathcal{Q}})$. This inclusion follows from the positive semidefiniteness of $L_{\mathcal{F}}$: any element of $\ker(L_{\mathcal{Q}})$, when embedded into the full cochain space, must lie in $\ker(L_{\mathcal{F}})$, and hence in $\ker(L_{SQ})$. Taking orthogonal complements and using the pseudoinverse $L_{\mathcal{Q}}^\dagger$ yields the result.

\begin{proof}
The dynamics \eqref{eq:stubborn-dynamics} form a linear inhomogeneous system
\[
\frac{dy}{dt} = -\alpha \left( L_{\mathcal{Q}} y + L_{QS} u \right)
\]
driven by the constant forcing $L_{QS} u$. Since $L_{\mathcal{F}}$ is positive semidefinite, so is its principal submatrix $L_{\mathcal{Q}}$. The equilibrium condition $L_{\mathcal{Q}} y = -L_{QS} u$ admits a solution provided $\operatorname{im}(L_{QS}) \subseteq \operatorname{im}(L_{\mathcal{Q}})$.

To verify this inclusion, let $z \in \ker(L_{\mathcal{Q}})$. Then
\begin{align*}
\langle \tilde{z}, L_{\mathcal{F}} \tilde{z} \rangle = z^T L_{\mathcal{Q}} z = 0.
\end{align*}
Since $L_{\mathcal{F}} \succeq 0$, this implies $\tilde{z} \in \ker(L_{\mathcal{F}})$, hence $L_{SQ} z = P_S L_{\mathcal{F}} \tilde{z} = 0$. Thus $\ker(L_{\mathcal{Q}}) \subseteq \ker(L_{SQ})$, and taking orthogonal complements yields $\operatorname{im}(L_{QS}) \subseteq \operatorname{im}(L_{\mathcal{Q}})$.

The explicit solution is
\begin{align*}
y(t) = e^{-t\alpha L_{\mathcal{Q}}} y(0) - L_{\mathcal{Q}}^\dagger \bigl(I - e^{-t\alpha L_{\mathcal{Q}}}\bigr) L_{QS} u.
\end{align*}
As $t \to \infty$, the matrix exponential $e^{-t\alpha L_{\mathcal{Q}}}$ converges to the projection $P_{H^0}$ onto $\ker(L_{\mathcal{Q}})$, yielding
\begin{align*}
y^\infty = P_{H^0}(y(0)) - L_{\mathcal{Q}}^\dagger L_{QS} u.
\end{align*}
Using $P_{H^0} = I - L_{\mathcal{Q}}^\dagger L_{\mathcal{Q}}$ and $L_{QS} u = P_{\mathcal{Q}} L_{\mathcal{F}} \tilde{u}$ gives the equivalent expression \eqref{eq:y_inf}.

For the variational characterization, observe that the Poisson equation $L_{\mathcal{Q}} y^\infty + L_{QS} u = 0$ is equivalent to $P_{\mathcal{Q}}(L_{\mathcal{F}} x^\infty) = 0$. This is the first-order condition for minimizing $\frac{1}{2}\|\delta x\|^2$ over the affine subspace $\{x : P_S(x) = u\}$: at a minimum, the gradient $L_{\mathcal{F}} x$ must be orthogonal to all feasible directions. Convexity ensures this critical point is the global minimum.
\end{proof}

\subsection{The Sheaf Poisson Equation}\label{ssec:poisson-equation}

The equilibrium condition \eqref{eq:poisson} deserves emphasis: it is a \emph{sheaf Poisson equation}, the natural analogue of the classical Poisson equation $\Delta u = f$ from potential theory. The Laplacian $L_{\mathcal{Q}}$ of the free sheaf plays the role of the differential operator, and the forcing term $L_{QS} u$ encodes how stubborn opinions propagate through the network to influence free agents.

The solution set of the Poisson equation forms an affine subspace
\begin{align*}
\mathcal{A} = \{y \in C^0(\mathcal{Q}) : L_{\mathcal{Q}} y = -L_{QS} u\}
\end{align*}
with direction space $H^0(G; \mathcal{Q}) = \ker(L_{\mathcal{Q}})$. The dynamics select a particular point in $\mathcal{A}$ according to the initial condition. Several equivalent expressions illuminate different aspects of this selection:
\begin{align}
y^\infty &= \operatorname*{argmin}_{y \in \mathcal{A}} \|y - y(0)\|, \label{eq:y_inf_argmin} \\
&= P_{H^0}(y(0)) - L_{\mathcal{Q}}^\dagger L_{QS} u, \label{eq:y_inf_proj} \\
&= y(0) - L_{\mathcal{Q}}^\dagger (L_{\mathcal{Q}} y(0) + L_{QS} u). \label{eq:y_inf_correction}
\end{align}
Expression \eqref{eq:y_inf_argmin} says that among all solutions to the Poisson equation, the dynamics choose the one closest to the initial condition. Expression \eqref{eq:y_inf_proj} decomposes the limit into two parts: the projection of the initial free opinion onto the cohomology $H^0(G; \mathcal{Q})$, which is preserved by the dynamics, plus the minimum-norm particular solution $-L_{\mathcal{Q}}^\dagger L_{QS} u$ that accounts for the stubborn forcing. Expression \eqref{eq:y_inf_correction} shows the limit as the initial condition minus the minimal correction needed to satisfy the constraint; agents adjust their flexible opinions only as much as necessary to accommodate the stubborn neighbors, a principle of minimal compromise.

The equilibrium depends on the initial condition only through its projection onto $H^0(G; \mathcal{Q})$. When this cohomology vanishes, the affine space $\mathcal{A}$ is a single point and the limit is unique. Otherwise, initial conditions differing by a global section of $\mathcal{Q}$ converge to distinct equilibria. The forced equilibrium structure connects to the theory of affine network sheaves developed by Gould~\cite{gould2024hilbert}, where inhomogeneous linear systems arise from affine restriction maps and solutions are characterized as projections onto affine cosets.

\begin{remark}[Relation to classical stubborn agents]
When stubborn agents have their entire stalks fixed ($S_v = \mathcal{F}(v)$ for $v \in U$ and $S_v = \{0\}$ otherwise), the sheaf of free opinions $\mathcal{Q}$ has trivial stalks at stubborn vertices and full stalks elsewhere. In this case, the Poisson equation reduces to the harmonic extension problem studied by Hansen and Ghrist~\cite{hansen2020}: find the opinion configuration on free agents that minimizes total disagreement given the boundary values imposed by stubborn agents. The cohomology $H^0(G; \mathcal{Q})$ coincides with the relative cohomology $H^0(G, U; \mathcal{F})$.
\end{remark}


\begin{figure}[ht]
\centering
\begin{tikzpicture}[
    vertex/.style={circle, fill=white, draw=black, thick, minimum size=5pt, inner sep=0pt},
    edge stalk h/.style={rectangle, rounded corners=3pt, fill=white, draw=black, thick, minimum width=16pt, minimum height=4pt},
    edge stalk v/.style={rectangle, rounded corners=3pt, fill=white, draw=black, thick, minimum width=4pt, minimum height=16pt},
    arrow/.style={-{Stealth[length=4pt]}, thick},
    every node/.style={font=\small},
    scale=0.95
]


\draw[thick] (0,0) -- (4.0,0) -- (4.0,3) -- (0,3) -- cycle;

\node[edge stalk h] (e_top_L) at (2.0,3) {};
\node[edge stalk h] (e_bot_L) at (2.0,0) {};
\node[edge stalk v] (e_left_L) at (0,1.5) {};
\node[edge stalk v] (e_right_L) at (4.0,1.5) {};

\node[vertex] (v1_L) at (0,3) {};
\node[vertex] (v2_L) at (4.0,3) {};
\node[vertex] (v3_L) at (4.0,0) {};
\node[vertex] (v4_L) at (0,0) {};

\node at (0.35,2.65) {$v_1$};
\node at (3.65,2.65) {$v_2$};
\node at (3.65,0.35) {$v_3$};
\node at (0.35,0.35) {$v_4$};

\node[above left] at (v1_L) {$\mathbb{R}$};
\node[above right] at (v2_L) {$\mathbb{R}^2$};
\node[below right] at (v3_L) {$\mathbb{R}$};
\node[below left] at (v4_L) {${\color{red}\mathbb{R}} \oplus \mathbb{R}$};

\node[anchor=base east, font=\footnotesize, inner sep=0pt] at (-0.15,4.2) {$x_{v_1}$};
\node[anchor=base, font=\footnotesize, inner sep=0pt] at (0,4.2) {$=$};
\node[anchor=base west, font=\footnotesize, inner sep=0pt] at (0.15,4.2) {$2$};
\node[anchor=base east, font=\footnotesize, inner sep=0pt] at (3.85,4.2) {$x_{v_2}$};
\node[anchor=base, font=\footnotesize, inner sep=0pt] at (4.0,4.2) {$=$};
\node[anchor=base west, font=\footnotesize, inner sep=0pt] at (4.15,4.2) {$\begin{pmatrix} 1 \\ -2 \end{pmatrix}$};
\node[anchor=base east, font=\footnotesize, inner sep=0pt] at (-0.15,-1.5) {$x_{v_4}$};
\node[anchor=base, font=\footnotesize, inner sep=0pt] at (0,-1.5) {$=$};
\node[anchor=base west, font=\footnotesize, inner sep=0pt] at (0.15,-1.5) {$\begin{pmatrix} {\color{red}1} \\ 0 \end{pmatrix}$};
\node[anchor=base east, font=\footnotesize, inner sep=0pt] at (3.85,-1.5) {$x_{v_3}$};
\node[anchor=base, font=\footnotesize, inner sep=0pt] at (4.0,-1.5) {$=$};
\node[anchor=base west, font=\footnotesize, inner sep=0pt] at (4.15,-1.5) {$0$};

\node[above=3pt] at (e_top_L) {$\mathbb{R}$};
\node[below=3pt] at (e_bot_L) {$\mathbb{R}$};
\node[left=3pt] at (e_left_L) {$\mathbb{R}^2$};
\node[right=3pt] at (e_right_L) {$\mathbb{R}$};

\draw[arrow] (0.15,3.3) -- (1.3,3.3);
\node[above, font=\scriptsize] at (0.7,3.3) {$[-2]$};

\draw[arrow] (-0.3,2.85) -- (-0.3,1.95);
\node[left, font=\scriptsize] at (-0.3,2.4) {$\begin{bmatrix} 1 \\ 0 \end{bmatrix}$};

\draw[arrow] (3.85,3.3) -- (2.7,3.3);
\node[above, font=\scriptsize] at (3.3,3.3) {$\begin{bmatrix} -1 & 2 \end{bmatrix}$};

\draw[arrow] (4.3,2.85) -- (4.3,1.95);
\node[right, font=\scriptsize] at (4.3,2.4) {$\begin{bmatrix} 1 & 1 \end{bmatrix}$};

\draw[arrow] (4.3,0.15) -- (4.3,1.05);
\node[right, font=\scriptsize] at (4.3,0.6) {$[1]$};

\draw[arrow] (3.85,-0.3) -- (2.7,-0.3);
\node[below, font=\scriptsize] at (3.3,-0.3) {$[-1]$};

\draw[arrow] (0.15,-0.3) -- (1.3,-0.3);
\node[below, font=\scriptsize] at (0.7,-0.3) {$\begin{bmatrix} 1 & -1 \end{bmatrix}$};

\draw[arrow] (-0.3,0.15) -- (-0.3,1.05);
\node[left, font=\scriptsize] at (-0.3,0.6) {$\begin{bmatrix} 1 & -1 \\ 1 & 0 \end{bmatrix}$};

\node[below=72pt] at (2.0,0) {\textbf{Initial state} $x(0)$};

\begin{scope}[xshift=10.0cm]

\draw[thick] (0,0) -- (4.0,0) -- (4.0,3) -- (0,3) -- cycle;

\node[edge stalk h] (e_top_R) at (2.0,3) {};
\node[edge stalk h] (e_bot_R) at (2.0,0) {};
\node[edge stalk v] (e_left_R) at (0,1.5) {};
\node[edge stalk v] (e_right_R) at (4.0,1.5) {};

\node[vertex] (v1_R) at (0,3) {};
\node[vertex] (v2_R) at (4.0,3) {};
\node[vertex] (v3_R) at (4.0,0) {};
\node[vertex] (v4_R) at (0,0) {};

\node at (0.35,2.65) {$v_1$};
\node at (3.65,2.65) {$v_2$};
\node at (3.65,0.35) {$v_3$};
\node at (0.35,0.35) {$v_4$};

\node[above left] at (v1_R) {$\mathbb{R}$};
\node[above right] at (v2_R) {$\mathbb{R}^2$};
\node[below right] at (v3_R) {$\mathbb{R}$};
\node[below left] at (v4_R) {${\color{red}\mathbb{R}} \oplus \mathbb{R}$};

\node[anchor=base east, font=\footnotesize, inner sep=0pt] at (-0.15,4.2) {$x^\infty_{v_1}$};
\node[anchor=base, font=\footnotesize, inner sep=0pt] at (0,4.2) {$=$};
\node[anchor=base west, font=\footnotesize, inner sep=0pt] at (0.15,4.2) {$\frac{5}{4}$};
\node[anchor=base east, font=\footnotesize, inner sep=0pt] at (3.85,4.2) {$x^\infty_{v_2}$};
\node[anchor=base, font=\footnotesize, inner sep=0pt] at (4.0,4.2) {$=$};
\node[anchor=base west, font=\footnotesize, inner sep=0pt] at (4.15,4.2) {$\begin{pmatrix} 0 \\ -\frac{5}{4} \end{pmatrix}$};
\node[anchor=base east, font=\footnotesize, inner sep=0pt] at (-0.15,-1.5) {$x^\infty_{v_4}$};
\node[anchor=base, font=\footnotesize, inner sep=0pt] at (0,-1.5) {$=$};
\node[anchor=base west, font=\footnotesize, inner sep=0pt] at (0.15,-1.5) {$\begin{pmatrix} {\color{red}1} \\ -\frac{1}{4} \end{pmatrix}$};
\node[anchor=base east, font=\footnotesize, inner sep=0pt] at (3.85,-1.5) {$x^\infty_{v_3}$};
\node[anchor=base, font=\footnotesize, inner sep=0pt] at (4.0,-1.5) {$=$};
\node[anchor=base west, font=\footnotesize, inner sep=0pt] at (4.15,-1.5) {$-\frac{5}{4}$};

\node[above=3pt] at (e_top_R) {$\mathbb{R}$};
\node[below=3pt] at (e_bot_R) {$\mathbb{R}$};
\node[left=3pt] at (e_left_R) {$\mathbb{R}^2$};
\node[right=3pt] at (e_right_R) {$\mathbb{R}$};

\draw[arrow] (0.15,3.3) -- (1.3,3.3);
\node[above, font=\scriptsize] at (0.7,3.3) {$[-2]$};

\draw[arrow] (-0.3,2.85) -- (-0.3,1.95);
\node[left, font=\scriptsize] at (-0.3,2.4) {$\begin{bmatrix} 1 \\ 0 \end{bmatrix}$};

\draw[arrow] (3.85,3.3) -- (2.7,3.3);
\node[above, font=\scriptsize] at (3.3,3.3) {$\begin{bmatrix} -1 & 2 \end{bmatrix}$};

\draw[arrow] (4.3,2.85) -- (4.3,1.95);
\node[right, font=\scriptsize] at (4.3,2.4) {$\begin{bmatrix} 1 & 1 \end{bmatrix}$};

\draw[arrow] (4.3,0.15) -- (4.3,1.05);
\node[right, font=\scriptsize] at (4.3,0.6) {$[1]$};

\draw[arrow] (3.85,-0.3) -- (2.7,-0.3);
\node[below, font=\scriptsize] at (3.3,-0.3) {$[-1]$};

\draw[arrow] (0.15,-0.3) -- (1.3,-0.3);
\node[below, font=\scriptsize] at (0.7,-0.3) {$\begin{bmatrix} 1 & -1 \end{bmatrix}$};

\draw[arrow] (-0.3,0.15) -- (-0.3,1.05);
\node[left, font=\scriptsize] at (-0.3,0.6) {$\begin{bmatrix} 1 & -1 \\ 1 & 0 \end{bmatrix}$};

\node[below=72pt] at (2.0,0) {\textbf{Equilibrium} $x^\infty$};

\end{scope}

\draw[-{Stealth[length=6pt]}, very thick] (5.0,1.5) -- (9.0,1.5);
\node[above, font=\footnotesize] at (7.0,1.5) {$\frac{dy}{dt} = -\alpha(L_{\mathcal{Q}} y + L_{QS} u)$};

\end{tikzpicture}
\caption{Constrained diffusion with a partially stubborn agent. Agent $v_4$ holds a fixed opinion in the first coordinate (red), while all other opinions evolve freely. \textbf{Left:} Initial state with $\|\delta x(0)\|^2 = 5$. \textbf{Right:} Equilibrium with $\|\delta x^\infty\|^2 = 1$. The stubborn direction at $v_4$ remains unchanged at $1$. Edges $e_{12}$, $e_{23}$, and $e_{34}$ achieve perfect agreement, but edge $e_{41}$ retains residual discrepancy $(0, -1)^T$ because $v_4$'s stubbornness prevents full consensus.}
\label{fig:stubborn-diffusion-example}
\end{figure}


\section{Stubborn Expressions}\label{sec:stubborn-expressions}

In the previous section, the discourse structure remained fixed while agents adapted their private opinions. Hansen and Ghrist \cite{hansen2020} introduced a dual perspective in their ``learning to lie'' dynamics: agents maintain fixed opinions but evolve the restriction maps that govern how those opinions are expressed. Under this dynamics, the coboundary operator $\delta$ evolves to minimize the discrepancy $\|\delta x\|^2$, and the system converges to the Frobenius nearest sheaf under which the fixed opinion distribution becomes a global section.

We now introduce a partial version of this structural adaptation. Rather than allowing all restriction maps to evolve, we designate a subset of incidence pairs whose restriction maps may adapt, while the remaining maps stay fixed. This models situations where communication norms vary across relationships: an agent may adapt how they express opinions to certain neighbors while following fixed patterns with others, due to formal hierarchies, longstanding conventions, or differing levels of social flexibility. Adaptation may be asymmetric: on a given edge, one endpoint may adjust its expression policy while the other maintains a rigid stance. We call this \emph{partial learning}, and the fixed restriction maps represent \emph{stubborn expressions}~\cite{kuran1997,ye2019expressedprivate}.

\subsection{Partial Learning Dynamics}\label{ssec:partial-learning}

To formalize partial learning, let $\mathcal{I} \subseteq \{(v,e) : v \unlhd e\}$ denote the set of adapting incidences; all other restriction maps remain fixed. Let $V_{\mathrm{maps}} = \bigoplus_{(v,e) \in \mathcal{I}} \mathrm{Hom}(\mathcal{F}(v), \mathcal{F}(e))$ be the space of adapting maps, equipped with the Frobenius inner product, and let $\rho \in V_{\mathrm{maps}}$ denote a configuration of adapting maps.

Fix a $0$-cochain $x \in C^0(G; \mathcal{F})$ representing the agents' opinions. Let $E' \subseteq E$ be the edges with at least one adapting incidence, and let $W = \bigoplus_{e \in E'} \mathcal{F}(e)$. Choosing an orientation for each edge, define the linear operator $A: V_{\mathrm{maps}} \to W$ by
\begin{align*}
(A\rho)_e = \sum_{(w,e) \in \mathcal{I}} \sigma_{w,e} \, \rho_{w \unlhd e}(x_w),
\end{align*}
where $\sigma_{w,e} = \pm 1$ is the incidence sign. Define the constant $c \in W$ encoding the contribution of fixed maps:
\begin{align*}
c_e = \sum_{(w,e) \notin \mathcal{I}} \sigma_{w,e} \, \mathcal{F}_{w \unlhd e}(x_w).
\end{align*}
The total discrepancy on adapting edges is $d(\rho) = A\rho + c$, which is affine in $\rho$.

\begin{theorem}[Partial learning: convergence to the nearest consistent sheaf]
\label{thm:partial-learning-unreg}
Let $\mathcal{F}$ be a discourse sheaf on a finite graph $G = (V,E)$ with finite-dimensional stalks. For $\beta > 0$, the gradient flow on the discrepancy energy $\mathcal{J}(\rho) = \frac{1}{2}\|A\rho + c\|^2$,
\begin{align}\label{eq:structure-dynamics}
\frac{d\rho}{dt} = -\beta A^T(A\rho(t) + c), \qquad \rho(0) = \rho_0,
\end{align}
admits a unique global solution and converges exponentially to
\begin{align}\label{eq:rho_inf}
\rho^\infty = \rho_0 - A^\dagger d_0,
\end{align}
where $d_0 = A\rho_0 + c$ is the initial discrepancy. The limiting restriction maps $\rho^\infty$, together with the fixed maps, define a sheaf $\mathcal{F}^\infty$ that is the Frobenius-nearest to $\mathcal{F}$ among all sheaves minimizing discrepancy on $E'$. If the system is consistent ($-c \in \mathrm{im}(A)$), then $\mathcal{F}^\infty$ achieves zero discrepancy on all adapting edges; otherwise, it minimizes the residual $\|A\rho + c\|$.
\end{theorem}

The proof proceeds by decomposing $V_{\mathrm{maps}}$ into $\ker(A)$ and its orthogonal complement. The component in $\ker(A)$ is preserved (variations invisible to discrepancy remain unchanged), while the component in $\ker(A)^\perp$ converges exponentially to the minimum-norm least-squares solution.

\begin{proof}
The dynamics $d{\rho}/dt = -\beta A^T A \rho - \beta A^T c$ form a linear system governed by the positive semidefinite matrix $M = A^T A$. Decompose $V_{\mathrm{maps}} = \ker(A) \oplus \ker(A)^\perp$. Since $\mathrm{im}(A^T) = \ker(A)^\perp$, the projection onto $\ker(A)$ satisfies $P_{\ker(A)} A^T = 0$, so
\begin{align*}
\frac{d}{dt} P_{\ker(A)} \rho = -\beta P_{\ker(A)} A^T(A\rho + c) = 0.
\end{align*}
Thus $P_{\ker(A)} \rho(t) = P_{\ker(A)} \rho_0$ for all $t \geq 0$.

On $\ker(A)^\perp$, the operator $M$ has strictly positive eigenvalues, so the dynamics converge exponentially to the unique solution of $A^T A \rho = -A^T c$ in this subspace. This solution is $-A^\dagger c$, the minimum-norm least-squares solution. Combining the two components yields
\begin{align*}
\rho^\infty = P_{\ker(A)}(\rho_0) - A^\dagger c = \rho_0 - A^\dagger(A\rho_0 + c) = \rho_0 - A^\dagger d_0,
\end{align*}
using $P_{\ker(A)} = I - A^\dagger A$. If $-c \in \mathrm{im}(A)$, then $A\rho^\infty + c = 0$.
\end{proof}


The limit $\rho^\infty$ admits equivalent characterizations that illuminate different aspects of the dynamics. Writing $\mathcal{A} = \{\rho : A^T(A\rho + c) = 0\}$ for the affine subspace of discrepancy minimizers,
\begin{align}
\rho^\infty &= \operatorname*{argmin}_{\rho \in \mathcal{A}} \|\rho - \rho_0\|, \label{eq:rho_inf_argmin} \\
&= P_{\ker(A)}(\rho_0) - A^\dagger c. \label{eq:rho_inf_proj}
\end{align}
Expression \eqref{eq:rho_inf_argmin} says that among all restriction maps achieving minimal discrepancy, the dynamics select the one closest to the initial configuration in the Frobenius norm. Expression \eqref{eq:rho_inf_proj} decomposes the limit into two parts: the projection of the initial maps onto $\ker(A)$, which is preserved by the dynamics, plus the minimum-norm particular solution $-A^\dagger c$ that accounts for the fixed maps. The equilibrium depends on the initial condition only through its projection onto $\ker(A)$; when $\ker(A) = \{0\}$, the limit is unique.

The form $\rho^\infty = \rho_0 - A^\dagger d_0$ from the theorem reveals the cost of accommodation: the term $A^\dagger d_0$ measures the minimal deviation from initial expression policies required to achieve harmony (or minimize disharmony) with the stubborn expressions. In asymmetric scenarios where one endpoint adapts while the other remains fixed, this cost falls entirely on the adapting agent, who must learn to translate their opinions into a language compatible with their stubborn neighbor.

\begin{remark}[Relation to classical learning to lie]
When all incidences adapt ($\mathcal{I} = \{(v,e) : v \unlhd e\}$), there are no fixed restriction maps and $c = 0$. The dynamics reduce to the learning to lie framework of Hansen and Ghrist \cite{hansen2020}, where the full coboundary operator evolves to minimize $\|\delta x\|^2$ and the limit $\rho^\infty = P_{\ker(A)}(\rho_0)$ is the Frobenius-nearest sheaf under which $x$ becomes a global section. Our framework generalizes this by allowing a designated subset of restriction maps to remain fixed.
\end{remark}


\begin{figure}[ht]
\centering
\begin{tikzpicture}[
    vertex/.style={circle, fill=white, draw=black, thick, minimum size=5pt, inner sep=0pt},
    edge stalk h/.style={rectangle, rounded corners=3pt, fill=white, draw=black, thick, minimum width=16pt, minimum height=4pt},
    edge stalk v/.style={rectangle, rounded corners=3pt, fill=white, draw=black, thick, minimum width=4pt, minimum height=16pt},
    arrow/.style={-{Stealth[length=4pt]}, thick},
    every node/.style={font=\small},
    scale=0.95
]


\draw[thick] (0,0) -- (4.0,0) -- (4.0,3) -- (0,3) -- cycle;

\node[edge stalk h] (e_top_L) at (2.0,3) {};
\node[edge stalk h] (e_bot_L) at (2.0,0) {};
\node[edge stalk v] (e_left_L) at (0,1.5) {};
\node[edge stalk v] (e_right_L) at (4.0,1.5) {};

\node[vertex] (v1_L) at (0,3) {};
\node[vertex] (v2_L) at (4.0,3) {};
\node[vertex] (v3_L) at (4.0,0) {};
\node[vertex] (v4_L) at (0,0) {};

\node at (0.35,2.65) {$v_1$};
\node at (3.65,2.65) {$v_2$};
\node at (3.65,0.35) {$v_3$};
\node at (0.35,0.35) {$v_4$};

\node[above left] at (v1_L) {$\mathbb{R}$};
\node[above right] at (v2_L) {$\mathbb{R}^2$};
\node[below right] at (v3_L) {$\mathbb{R}$};
\node[below left] at (v4_L) {$\mathbb{R}^2$};

\node[anchor=base east, font=\footnotesize, inner sep=0pt] at (-0.15,4.2) {$x_{v_1}$};
\node[anchor=base, font=\footnotesize, inner sep=0pt] at (0,4.2) {$=$};
\node[anchor=base west, font=\footnotesize, inner sep=0pt] at (0.15,4.2) {$2$};
\node[anchor=base east, font=\footnotesize, inner sep=0pt] at (3.85,4.2) {$x_{v_2}$};
\node[anchor=base, font=\footnotesize, inner sep=0pt] at (4.0,4.2) {$=$};
\node[anchor=base west, font=\footnotesize, inner sep=0pt] at (4.15,4.2) {$\begin{pmatrix} 1 \\ -2 \end{pmatrix}$};
\node[anchor=base east, font=\footnotesize, inner sep=0pt] at (-0.15,-1.5) {$x_{v_4}$};
\node[anchor=base, font=\footnotesize, inner sep=0pt] at (0,-1.5) {$=$};
\node[anchor=base west, font=\footnotesize, inner sep=0pt] at (0.15,-1.5) {$\begin{pmatrix} 1 \\ 0 \end{pmatrix}$};
\node[anchor=base east, font=\footnotesize, inner sep=0pt] at (3.85,-1.5) {$x_{v_3}$};
\node[anchor=base, font=\footnotesize, inner sep=0pt] at (4.0,-1.5) {$=$};
\node[anchor=base west, font=\footnotesize, inner sep=0pt] at (4.15,-1.5) {$0$};

\node[above=3pt] at (e_top_L) {$\mathbb{R}$};
\node[below=3pt] at (e_bot_L) {$\mathbb{R}$};
\node[left=3pt] at (e_left_L) {$\mathbb{R}^2$};
\node[right=3pt] at (e_right_L) {$\mathbb{R}$};

\draw[arrow] (0.15,3.3) -- (1.3,3.3);
\node[above, font=\scriptsize] at (0.7,3.3) {$[-2]$};

\draw[arrow] (-0.3,2.85) -- (-0.3,1.95);
\node[left, font=\scriptsize] at (-0.3,2.4) {$\begin{bmatrix} 1 \\ 0 \end{bmatrix}$};

\draw[arrow] (3.85,3.3) -- (2.7,3.3);
\node[above, font=\scriptsize] at (3.3,3.3) {$\begin{bmatrix} -1 & 2 \end{bmatrix}$};

\draw[arrow] (4.3,2.85) -- (4.3,1.95);
\node[right, font=\scriptsize] at (4.3,2.4) {$\begin{bmatrix} 1 & 1 \end{bmatrix}$};

\draw[arrow] (4.3,0.15) -- (4.3,1.05);
\node[right, font=\scriptsize] at (4.3,0.6) {$[1]$};

\draw[arrow] (3.85,-0.3) -- (2.7,-0.3);
\node[below, font=\scriptsize] at (3.3,-0.3) {$[-1]$};

\draw[arrow, red] (0.15,-0.3) -- (1.3,-0.3);
\node[below, font=\scriptsize, red] at (0.7,-0.3) {$\begin{bmatrix} 1 & -1 \end{bmatrix}$};

\draw[arrow, red] (-0.3,0.15) -- (-0.3,1.05);
\node[left, font=\scriptsize, red] at (-0.3,0.6) {$\begin{bmatrix} 1 & -1 \\ 1 & 0 \end{bmatrix}$};

\node[below=72pt] at (2.0,0) {\textbf{Initial state} $\rho(0)$};

\begin{scope}[xshift=10.0cm]

\draw[thick] (0,0) -- (4.0,0) -- (4.0,3) -- (0,3) -- cycle;

\node[edge stalk h] (e_top_R) at (2.0,3) {};
\node[edge stalk h] (e_bot_R) at (2.0,0) {};
\node[edge stalk v] (e_left_R) at (0,1.5) {};
\node[edge stalk v] (e_right_R) at (4.0,1.5) {};

\node[vertex] (v1_R) at (0,3) {};
\node[vertex] (v2_R) at (4.0,3) {};
\node[vertex] (v3_R) at (4.0,0) {};
\node[vertex] (v4_R) at (0,0) {};

\node at (0.35,2.65) {$v_1$};
\node at (3.65,2.65) {$v_2$};
\node at (3.65,0.35) {$v_3$};
\node at (0.35,0.35) {$v_4$};

\node[above left] at (v1_R) {$\mathbb{R}$};
\node[above right] at (v2_R) {$\mathbb{R}^2$};
\node[below right] at (v3_R) {$\mathbb{R}$};
\node[below left] at (v4_R) {$\mathbb{R}^2$};

\node[anchor=base east, font=\footnotesize, inner sep=0pt] at (-0.15,4.2) {$x_{v_1}$};
\node[anchor=base, font=\footnotesize, inner sep=0pt] at (0,4.2) {$=$};
\node[anchor=base west, font=\footnotesize, inner sep=0pt] at (0.15,4.2) {$2$};
\node[anchor=base east, font=\footnotesize, inner sep=0pt] at (3.85,4.2) {$x_{v_2}$};
\node[anchor=base, font=\footnotesize, inner sep=0pt] at (4.0,4.2) {$=$};
\node[anchor=base west, font=\footnotesize, inner sep=0pt] at (4.15,4.2) {$\begin{pmatrix} 1 \\ -2 \end{pmatrix}$};
\node[anchor=base east, font=\footnotesize, inner sep=0pt] at (-0.15,-1.5) {$x_{v_4}$};
\node[anchor=base, font=\footnotesize, inner sep=0pt] at (0,-1.5) {$=$};
\node[anchor=base west, font=\footnotesize, inner sep=0pt] at (0.15,-1.5) {$\begin{pmatrix} 1 \\ 0 \end{pmatrix}$};
\node[anchor=base east, font=\footnotesize, inner sep=0pt] at (3.85,-1.5) {$x_{v_3}$};
\node[anchor=base, font=\footnotesize, inner sep=0pt] at (4.0,-1.5) {$=$};
\node[anchor=base west, font=\footnotesize, inner sep=0pt] at (4.15,-1.5) {$0$};

\node[above=3pt] at (e_top_R) {$\mathbb{R}$};
\node[below=3pt] at (e_bot_R) {$\mathbb{R}$};
\node[left=3pt] at (e_left_R) {$\mathbb{R}^2$};
\node[right=3pt] at (e_right_R) {$\mathbb{R}$};

\draw[arrow] (0.15,3.3) -- (1.3,3.3);
\node[above, font=\scriptsize] at (0.7,3.3) {$[-\frac{20}{9}]$};

\draw[arrow] (-0.3,2.85) -- (-0.3,1.95);
\node[left, font=\scriptsize] at (-0.3,2.4) {$\begin{bmatrix} \frac{1}{2} \\ \frac{1}{2} \end{bmatrix}$};

\draw[arrow] (3.85,3.3) -- (2.7,3.3);
\node[above, font=\scriptsize] at (3.3,3.3) {$\begin{bmatrix} -\frac{8}{9} & \frac{16}{9} \end{bmatrix}$};

\draw[arrow] (4.3,2.85) -- (4.3,1.95);
\node[right, font=\scriptsize] at (4.3,2.4) {$\begin{bmatrix} \frac{6}{5} & \frac{3}{5} \end{bmatrix}$};

\draw[arrow] (4.3,0.15) -- (4.3,1.05);
\node[right, font=\scriptsize] at (4.3,0.6) {$[1]$};

\draw[arrow] (3.85,-0.3) -- (2.7,-0.3);
\node[below, font=\scriptsize] at (3.3,-0.3) {$[-1]$};

\draw[arrow, red] (0.15,-0.3) -- (1.3,-0.3);
\node[below, font=\scriptsize, red] at (0.7,-0.3) {$\begin{bmatrix} 1 & -1 \end{bmatrix}$};

\draw[arrow, red] (-0.3,0.15) -- (-0.3,1.05);
\node[left, font=\scriptsize, red] at (-0.3,0.6) {$\begin{bmatrix} 1 & -1 \\ 1 & 0 \end{bmatrix}$};

\node[below=72pt] at (2.0,0) {\textbf{Equilibrium} $\rho^\infty$};

\end{scope}

\draw[-{Stealth[length=6pt]}, very thick] (5.0,1.5) -- (9.0,1.5);
\node[above, font=\footnotesize] at (7.0,1.5) {$\frac{d\rho}{dt} = -\beta A^T(A\rho + c)$};

\end{tikzpicture}
\caption{Partial learning with stubborn expressions. Opinions $x$ are fixed throughout; restriction maps from $v_4$ (red) are frozen while the remaining maps (black) adapt. \textbf{Left:} Initial configuration with $\|\delta x\|^2 = 5$. \textbf{Right:} Equilibrium with $\|\delta x\|^2 = 1$. Agents $v_1$, $v_2$, $v_3$ adjust their expression policies to accommodate $v_4$'s rigid communication. Edge $e_{34}$ retains residual discrepancy $1$ because $x_{v_3} = 0$ makes $v_3$'s expressed opinion vanish regardless of its restriction map, while $v_4$'s fixed map expresses $1$.}
\label{fig:stubborn-expressions-example}
\end{figure}
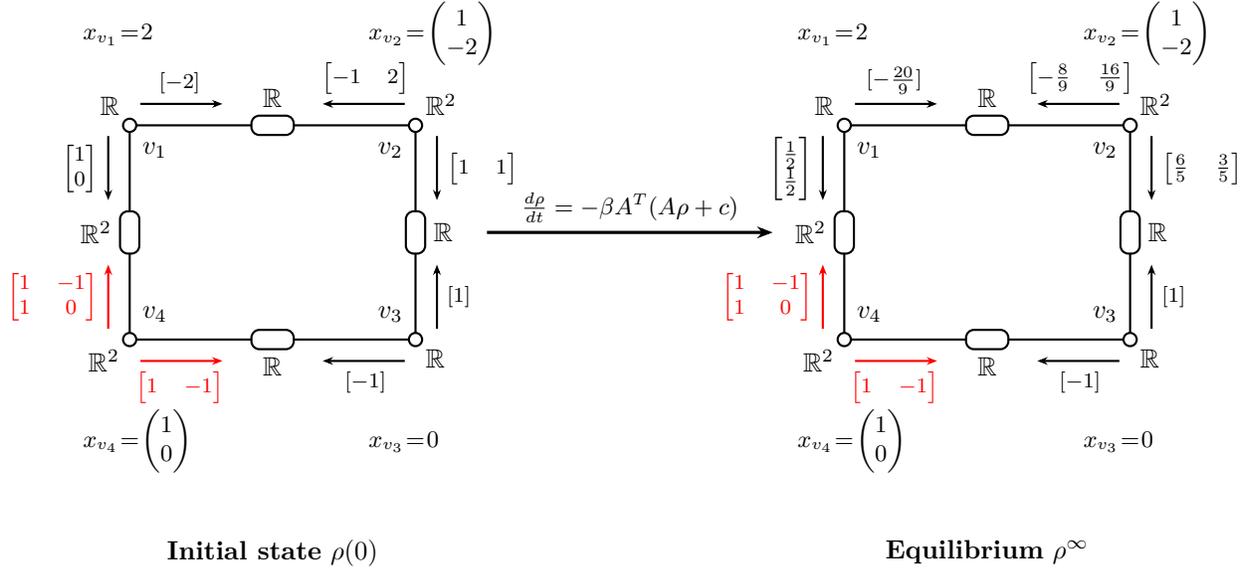


In practice, solving the dynamics requires working with the matrix $A^T A$, which may be ill-conditioned or singular. This occurs when the edge constraints are redundant, or when the fixed opinions $x$ are small, noisy, or collinear: one cannot robustly infer communication rules from degenerate data. Adding a regularization term penalizing deviations from the initial restriction maps ensures a unique equilibrium and stabilizes computation.

\begin{theorem}[Regularized partial learning]
\label{thm:partial-learning-reg}
Fix $\lambda > 0$ and consider the regularized objective
\begin{align}
\mathcal{L}_\lambda(\rho) = \tfrac{1}{2}\|A\rho + c\|^2 + \tfrac{\lambda}{2}\|\rho - \rho_0\|^2.
\end{align}
The gradient flow $d\rho/dt  = -\beta\nabla \mathcal{L}_\lambda(\rho)$ converges exponentially to the unique minimizer $\rho^*$, characterized by
\begin{align}
(A^T A + \lambda I)\rho^* = \lambda \rho_0 - A^T c.
\end{align}
As $\lambda \to 0$, the minimizers $\rho^*(\lambda)$ converge to the unregularized limit $\rho^\infty$ from Theorem~\ref{thm:partial-learning-unreg}.
\end{theorem}

\begin{proof}
The gradient is $\nabla \mathcal{L}_\lambda(\rho) = A^T(A\rho + c) + \lambda(\rho - \rho_0)$, so the dynamics become $d\rho/dt  = -\beta(A^T A + \lambda I)\rho + \beta(\lambda \rho_0 - A^T c)$. Since $A^T A + \lambda I$ is strictly positive definite, this linear system converges exponentially to its unique equilibrium. As $\lambda \to 0$, standard Tikhonov regularization theory gives convergence to the minimum-norm least-squares solution closest to $\rho_0$.
\end{proof}

The regularization term $\frac{\lambda}{2}\|\rho - \rho_0\|^2$ models a \emph{cost of adaptation}: agents prefer to maintain established expression patterns unless disagreement pressure justifies the effort of change. Large $\lambda$ means expressive habits resist change; small $\lambda$ favors harmony over consistency. The equilibrium $\rho^*$ balances pressure for expressive stability against the drive toward social accommodation.

\begin{remark}[Relation to weighted reluctance]
Adding regularization to opinion dynamics (when expressions do not evolve) is a particular case of the weighted reluctance framework of Hansen and Ghrist \cite{hansen2020}, where a term $\frac{\gamma}{2}\|y - y_0\|^2$ penalizes deviation from initial beliefs. The two regularizations capture dual phenomena: reluctance to change one's mind versus reluctance to change one's mode of expression.
\end{remark}

\subsection{The Sheaf of Free Structures}\label{ssec:sheaf-free-structures}

A reader comparing the proofs of Theorem~\ref{thm:stubborn-directions} and Theorem~\ref{thm:partial-learning-unreg} will notice they are structurally identical: both describe gradient descent on a quadratic energy, both converge to solutions of Poisson-type equations, and both characterize equilibria as projections onto affine cosets. This suggests a sheaf-theoretic framework unifying the two settings.

The underlying unity stems from the bilinearity of the discrepancy map. For an edge $e = u \sim v$ and a discourse sheaf $\mathcal{F}$, the local disagreement is
\[
D_e(\mathcal{F}, x) = \mathcal{F}_{v \unlhd e}(x_v) - \mathcal{F}_{u \unlhd e}(x_u).
\]
This expression is linear in $x$ when $\mathcal{F}$ is fixed (Section~\ref{sec:stubborn-opinions}), and linear in the restriction maps of $\mathcal{F}$ when $x$ is fixed (the present section). We now construct an auxiliary sheaf that makes this parallel precise.

\begin{definition}[Sheaf of Structures]
\label{def:sheaf-of-structures}
Let $\mathcal{F}$ be a discourse sheaf on $G$ and let $x \in C^0(G; \mathcal{F})$ be a fixed $0$-cochain. The \emph{sheaf of structures} $\mathcal{H}^x$ on $G$ is defined as follows:
\begin{itemize}
\item The vertex stalk at $v \in V$ is the space of all candidate restriction maps from $v$:
\[
\mathcal{H}^x(v) = \bigoplus_{e : v \unlhd e} \mathrm{Hom}(\mathcal{F}(v), \mathcal{F}(e)).
\]
\item The edge stalk at $e \in E$ is the discourse space: $\mathcal{H}^x(e) = \mathcal{F}(e)$.
\item For each incidence $v \unlhd e$, the restriction map $(\mathcal{H}^x)_{v \unlhd e}: \mathcal{H}^x(v) \to \mathcal{H}^x(e)$ evaluates the $e$-component on the fixed opinion:
\[
(\mathcal{H}^x)_{v \unlhd e}\bigl(\rho\bigr) = \mathrm{ev}_{x_v} \circ \pi_e (\rho) = \rho_{v \unlhd e}(x_v),
\]
where $\rho = (\rho_{v \unlhd e'})_{e' : v \unlhd e'} \in \mathcal{H}^x(v)$ is a tuple of linear maps indexed by the edges incident to $v$, $\pi_e: \mathcal{H}^x(v) \to \mathrm{Hom}(\mathcal{F}(v), \mathcal{F}(e))$ projects onto the $e$-component, and $\mathrm{ev}_{x_v}(\rho) = \rho(x_v)$ evaluates at the fixed opinion.
\end{itemize}
\end{definition}

A $0$-cochain in $\mathcal{H}^x$ assigns a candidate restriction map to every incidence, specifying a complete sheaf structure over the stalks of $\mathcal{F}$. The coboundary measures discrepancy: $(\delta^{\mathcal{H}^x} \rho)_e = 0$ precisely when the maps encoded in $\rho$ make the fixed opinions $x$ publicly consistent on edge $e$. Global sections of $\mathcal{H}^x$ are thus sheaf structures for which $x$ is a global section. The same construction appears in Gould~\cite{gould2024hilbert} under the name \emph{restriction map diffusion} in the context of cellular sheaves of Hilbert spaces, and is implicit in the ``dual diffusion'' formulation of Hernandez Caralt et al.~\cite{hernandezcaralt2024joint}. A detailed discussion of the categorical underpinnings of this construction, including its generality, the precise sense in which $\mathcal{F}$ and $\mathcal{H}^x$ are \textit{dual}, and a canonical embedding relating them, appears in Appendix~\ref{app:categorical-structure}.

Now let $\mathcal{I} \subseteq \{(v, e) : v \unlhd e\}$ be the set of adapting incidences. Each vertex stalk decomposes as
\[
\mathcal{H}^x(v) = \underbrace{\bigoplus_{e : (v,e) \notin \mathcal{I}} \mathrm{Hom}(\mathcal{F}(v), \mathcal{F}(e))}_{S_v^{\mathcal{H}}} \oplus \underbrace{\bigoplus_{e : (v,e) \in \mathcal{I}} \mathrm{Hom}(\mathcal{F}(v), \mathcal{F}(e))}_{T_v^{\mathcal{H}}},
\]
where $S_v^{\mathcal{H}}$ contains the stubborn restriction maps and $T_v^{\mathcal{H}}$ contains the adapting restriction maps, in direct analogy with the decomposition $\mathcal{F}(v) = S_v \oplus T_v$ of Section~\ref{sec:stubborn-opinions}.

\begin{definition}[Sheaf of Free Structures]
\label{def:sheaf-free-structures}
The \emph{sheaf of free structures} $\mathcal{H}^x_{\mathcal{I}}$ is the subsheaf of $\mathcal{H}^x$ with vertex stalks $\mathcal{H}^x_{\mathcal{I}}(v) = T_v^{\mathcal{H}}$. The edge stalks are $\mathcal{H}^x_{\mathcal{I}}(e) = \mathcal{F}(e)$ for $e \in E'$ and $\mathcal{H}^x_{\mathcal{I}}(e) = 0$ for $e \notin E'$, where $E'$ denotes the edges with at least one adapting incidence. The restriction maps are inherited from $\mathcal{H}^x$.
\end{definition}

Note that $\mathcal{H}^x_{\mathcal{I}}$ is a well-defined cellular sheaf on the full graph $G$, not merely on the subgraph induced by $E'$. The trivial edge stalks for $e \notin E'$ ensure that the coboundary operator $\delta^{\mathcal{H}^x_{\mathcal{I}}}$ and Laplacian $L_{\mathcal{H}^x_{\mathcal{I}}}$ are defined on all of $G$, with contributions from edges outside $E'$ vanishing automatically. The subsheaf inclusion $\mathcal{H}^x_{\mathcal{I}} \hookrightarrow \mathcal{H}^x$ induces a corresponding decomposition of the cochain space: $C^0(G; \mathcal{H}^x) = C^0(S^{\mathcal{H}}) \oplus C^0(\mathcal{H}^x_{\mathcal{I}})$, where $C^0(S^{\mathcal{H}}) = \bigoplus_v S_v^{\mathcal{H}}$ and $C^0(\mathcal{H}^x_{\mathcal{I}}) = \bigoplus_v T_v^{\mathcal{H}} = V_{\mathrm{maps}}$.

\begin{lemma}[Cochain structure of $\mathcal{H}^x_{\mathcal{I}}$]
\label{lem:cochain-structure}
The cochain spaces of $\mathcal{H}^x_{\mathcal{I}}$ admit the following descriptions:
\begin{enumerate}
\item The $0$-cochain space is $C^0(G; \mathcal{H}^x_{\mathcal{I}}) = \bigoplus_{v \in V} T_v^{\mathcal{H}} = V_{\mathrm{maps}}$, with inner product
\[
\langle \rho, \rho' \rangle_{C^0} = \sum_{v \in V} \sum_{e : (v,e) \in \mathcal{I}} \langle \rho_{v \unlhd e}, \rho'_{v \unlhd e} \rangle_F,
\]
where $\langle \cdot, \cdot \rangle_F$ denotes the Frobenius inner product on $\mathrm{Hom}(\mathcal{F}(v), \mathcal{F}(e))$.
\item The $1$-cochain space is $C^1(G; \mathcal{H}^x_{\mathcal{I}}) = \bigoplus_{e \in E} \mathcal{H}^x_{\mathcal{I}}(e) = \bigoplus_{e \in E'} \mathcal{F}(e) = W$, since $\mathcal{H}^x_{\mathcal{I}}(e) = 0$ for $e \notin E'$. The inner product is inherited from the stalks of $\mathcal{F}$.
\item Under these inner products, $\delta^{\mathcal{H}^x_{\mathcal{I}}} = A$ and its adjoint is $(\delta^{\mathcal{H}^x_{\mathcal{I}}})^* = A^T$.
\end{enumerate}
\end{lemma}

\begin{proof}
Items (1) and (2) follow directly from Definition~\ref{def:sheaf-free-structures}. For item (3), we verify that the coboundary formula matches the definition of $A$. For an oriented edge $e = u \to v$, the standard coboundary formula gives
\[
(\delta^{\mathcal{H}^x_{\mathcal{I}}} \rho)_e = (\mathcal{H}^x_{\mathcal{I}})_{v \unlhd e}(\rho_v) - (\mathcal{H}^x_{\mathcal{I}})_{u \unlhd e}(\rho_u).
\]
The restriction map $(\mathcal{H}^x_{\mathcal{I}})_{w \unlhd e}$ is nonzero only when $(w,e) \in \mathcal{I}$, in which case it evaluates the $e$-component at $x_w$: $(\mathcal{H}^x_{\mathcal{I}})_{w \unlhd e}(\rho_w) = \rho_{w \unlhd e}(x_w)$. The head vertex $v$ receives sign $+1$ and the tail vertex $u$ receives sign $-1$, matching the convention $\sigma_{w,e} = +1$ for $w = v$ and $\sigma_{w,e} = -1$ for $w = u$. Thus
\[
(\delta^{\mathcal{H}^x_{\mathcal{I}}} \rho)_e = \sum_{(w,e) \in \mathcal{I}} \sigma_{w,e} \cdot \rho_{w \unlhd e}(x_w) = (A\rho)_e.
\]
The adjoint formula follows from the identity $\langle Mx, y \rangle = \langle M, yx^T \rangle_F$ for matrices: given $\rho \in V_{\mathrm{maps}}$ and $y \in W$,
\[
\langle A\rho, y \rangle_W = \sum_{(w,e) \in \mathcal{I}} \sigma_{w,e} \langle \rho_{w \unlhd e}(x_w), y_e \rangle = \sum_{(w,e) \in \mathcal{I}} \sigma_{w,e} \langle \rho_{w \unlhd e}, y_e x_w^T \rangle_F = \langle \rho, A^T y \rangle_{V_{\mathrm{maps}}},
\]
so $(A^T y)_{w \unlhd e} = \sigma_{w,e} \cdot y_e x_w^T$ for each $(w,e) \in \mathcal{I}$.
\end{proof}


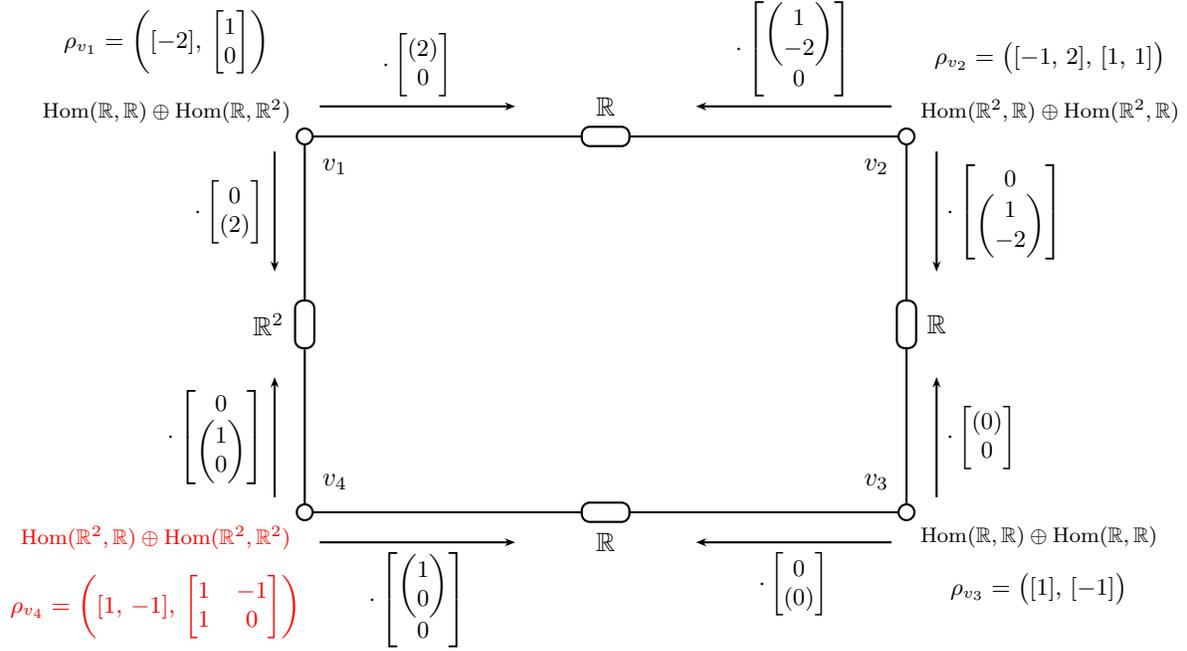
\begin{figure}[ht]
\centering
\begin{tikzpicture}[
    vertex/.style={circle, fill=white, draw=black, thick, minimum size=6pt, inner sep=0pt},
    edge stalk h/.style={rectangle, rounded corners=3pt, fill=white, draw=black, thick, minimum width=18pt, minimum height=5pt},
    edge stalk v/.style={rectangle, rounded corners=3pt, fill=white, draw=black, thick, minimum width=5pt, minimum height=18pt},
    arrow/.style={-{Stealth[length=4pt]}, thick},
    every node/.style={font=\small},
    scale=1.0
]

\draw[thick] (0,0) -- (8,0) -- (8,5) -- (0,5) -- cycle;

\node[edge stalk h] (e_top) at (4,5) {};
\node[edge stalk h] (e_bot) at (4,0) {};
\node[edge stalk v] (e_left) at (0,2.5) {};
\node[edge stalk v] (e_right) at (8,2.5) {};

\node[vertex] (v1) at (0,5) {};
\node[vertex] (v2) at (8,5) {};
\node[vertex] (v3) at (8,0) {};
\node[vertex] (v4) at (0,0) {};

\node at (0.4,4.6) {$v_1$};
\node at (7.6,4.6) {$v_2$};
\node at (7.6,0.4) {$v_3$};
\node at (0.4,0.4) {$v_4$};

\node[above left=2pt, font=\scriptsize, align=right] (hom1) at (v1) {$\mathrm{Hom}(\mathbb{R},\mathbb{R}) \oplus \mathrm{Hom}(\mathbb{R},\mathbb{R}^2)$};
\node[above right=2pt, font=\scriptsize, align=left] (hom2) at (v2) {$\mathrm{Hom}(\mathbb{R}^2,\mathbb{R}) \oplus \mathrm{Hom}(\mathbb{R}^2,\mathbb{R})$};
\node[below right=2pt, font=\scriptsize, align=left] (hom3) at (v3) {$\mathrm{Hom}(\mathbb{R},\mathbb{R}) \oplus \mathrm{Hom}(\mathbb{R},\mathbb{R})$};
\node[below left=2pt, font=\scriptsize, align=right, red] (hom4) at (v4) {$\mathrm{Hom}(\mathbb{R}^2,\mathbb{R}) \oplus \mathrm{Hom}(\mathbb{R}^2,\mathbb{R}^2)$};

\node[above=2pt, font=\footnotesize, anchor=south] at (hom1.north) {$\rho_{v_1} = \left( [-2],\, \begin{bmatrix} 1 \\ 0 \end{bmatrix} \right)$};
\node[above=2pt, font=\footnotesize, anchor=south] at (hom2.north) {$\rho_{v_2} = \bigl( [-1,\, 2],\, [1,\, 1] \bigr)$};
\node[below=2pt, font=\footnotesize, anchor=north] at (hom3.south) {$\rho_{v_3} = \bigl( [1],\, [-1] \bigr)$};
\node[below=2pt, font=\footnotesize, anchor=north, red] at (hom4.south) {$\rho_{v_4} = \left( [1,\, {-}1],\, \begin{bmatrix} 1 & -1 \\ 1 & 0 \end{bmatrix} \right)$};

\node[above=4pt, font=\small] at (e_top) {$\mathbb{R}$};
\node[below=4pt, font=\small] at (e_bot) {$\mathbb{R}$};
\node[left=4pt, font=\small] at (e_left) {$\mathbb{R}^2$};
\node[right=4pt, font=\small] at (e_right) {$\mathbb{R}$};


\draw[arrow] (0.2,5.4) -- (2.8,5.4);
\node[above, font=\footnotesize] at (1.5,5.4) {$\cdot\begin{bmatrix} (2) \\ 0 \end{bmatrix}$};

\draw[arrow] (-0.4,4.8) -- (-0.4,3.2);
\node[left, font=\footnotesize] at (-0.4,4.0) {$\cdot\begin{bmatrix} 0 \\ (2) \end{bmatrix}$};

\draw[arrow] (7.8,5.4) -- (5.2,5.4);
\node[above, font=\footnotesize] at (6.5,5.4) {$\cdot\begin{bmatrix} \begin{pmatrix} 1 \\ -2 \end{pmatrix} \\ 0 \end{bmatrix}$};

\draw[arrow] (8.4,4.8) -- (8.4,3.2);
\node[right, font=\footnotesize] at (8.4,4.0) {$\cdot\begin{bmatrix} 0 \\ \begin{pmatrix} 1 \\ -2 \end{pmatrix} \end{bmatrix}$};

\draw[arrow] (8.4,0.2) -- (8.4,1.8);
\node[right, font=\footnotesize] at (8.4,1.0) {$\cdot\begin{bmatrix} (0) \\ 0 \end{bmatrix}$};

\draw[arrow] (7.8,-0.4) -- (5.2,-0.4);
\node[below, font=\footnotesize] at (6.5,-0.4) {$\cdot\begin{bmatrix} 0 \\ (0) \end{bmatrix}$};

\draw[arrow] (0.2,-0.4) -- (2.8,-0.4);
\node[below, font=\footnotesize] at (1.5,-0.4) {$\cdot\begin{bmatrix} \begin{pmatrix} 1 \\ 0 \end{pmatrix} \\ 0 \end{bmatrix}$};

\draw[arrow] (-0.4,0.2) -- (-0.4,1.8);
\node[left, font=\footnotesize] at (-0.4,1.0) {$\cdot\begin{bmatrix} 0 \\ \begin{pmatrix} 1 \\ 0 \end{pmatrix} \end{bmatrix}$};

\end{tikzpicture}
\caption{The sheaf of structures $\mathcal{H}^x$ with initial $0$-cochain $\rho(0)$, corresponding to the example in Figure~\ref{fig:stubborn-expressions-example}. Vertex stalks are direct sums of Hom spaces; the stalk at $v_4$ (red) is fixed. The $0$-cochains $\rho_v$ are tuples of matrices corresponding to restriction maps in the original discourse sheaf $\mathcal{F}$. The restriction maps in $\mathcal{H}^x$ are evaluation maps at the fixed opinions $x$: the notation $\cdot\begin{bmatrix} (x_v) \\ 0 \end{bmatrix}$ extracts the first component and right-multiplies by $x_v$, while $\cdot\begin{bmatrix} 0 \\ (x_v) \end{bmatrix}$ extracts the second. Note that $x_{v_3} = 0$, so the restriction maps from $v_3$ annihilate any $0$-cochain, creating an unavoidable obstruction to consensus on edge $e_{34}$.}
\label{fig:sheaf-admissible-structures}
\end{figure}

\begin{proposition}[Cohomological Characterization]
\label{prop:structure-cohomology}
Under the identification $V_{\mathrm{maps}} \cong C^0(G; \mathcal{H}^x_{\mathcal{I}})$:
\begin{enumerate}
\item The coboundary of $\mathcal{H}^x_{\mathcal{I}}$ coincides with the discrepancy operator: $\delta^{\mathcal{H}^x_{\mathcal{I}}} = A$.
\item The Laplacian of $\mathcal{H}^x_{\mathcal{I}}$ is the structure Laplacian: $L_{\mathcal{H}^x_{\mathcal{I}}} = A^T A$.
\item The forcing term $c$ from Theorem~\ref{thm:partial-learning-unreg} arises from the stubborn configuration: $c_e = (\delta^{\mathcal{H}^x} \tilde{\sigma}_0)_e$ for each $e \in E'$, where $\tilde{\sigma}_0 = \iota_{S^{\mathcal{H}}}(\sigma_0)$ embeds the fixed restriction maps into $C^0(G; \mathcal{H}^x)$.
\end{enumerate}
\end{proposition}

\begin{proof}
Item (1) follows from Lemma~\ref{lem:cochain-structure}: the sign conventions in the coboundary formula match those defining $A$, so $\delta^{\mathcal{H}^x_{\mathcal{I}}} = A$. Item (2) follows from the adjoint relation established in the lemma: $L_{\mathcal{H}^x_{\mathcal{I}}} = (\delta^{\mathcal{H}^x_{\mathcal{I}}})^* \delta^{\mathcal{H}^x_{\mathcal{I}}} = A^T A$. For item (3), the embedded stubborn configuration $\tilde{\sigma}_0$ has $(\tilde{\sigma}_0)_{w \unlhd e} = \mathcal{F}_{w \unlhd e}$ for $(w,e) \notin \mathcal{I}$ and zero otherwise. Applying the coboundary $\delta^{\mathcal{H}^x}$ of the full sheaf gives
\[
(\delta^{\mathcal{H}^x} \tilde{\sigma}_0)_e = \sum_{(w,e) \notin \mathcal{I}} \sigma_{w,e} \cdot \mathcal{F}_{w \unlhd e}(x_w) = c_e.
\]
\end{proof}

The partial learning dynamics of Theorem~\ref{thm:partial-learning-unreg} is therefore sheaf Laplacian diffusion on $\mathcal{H}^x_{\mathcal{I}}$, and the equilibrium equation $A^T A \rho^\infty = -A^T c$ is the sheaf Poisson equation with forcing from the stubborn structures. The parallel with Section~\ref{sec:stubborn-opinions} is now exact: opinions evolve on $\mathcal{Q}$ with forcing from stubborn opinions, while structures evolve on $\mathcal{H}^x_{\mathcal{I}}$ with forcing from stubborn structures. The correspondence is summarized below.

\begin{table}[ht]
\centering
\caption{Analogy between stubborn opinions and stubborn expressions.}
\begin{tabular}{lcc}
\toprule
& \textbf{Stubborn Opinions} & \textbf{Stubborn Expressions} \\
\midrule
\textbf{Fixed data} 
& Sheaf $\mathcal{F}$, stubborn values $u$ 
& Opinions $x$, stubborn maps \\
\textbf{Variable} 
& Free opinions $y \in C^0(\mathcal{Q})$ 
& Adapting maps $\rho \in C^0(\mathcal{H}^x_{\mathcal{I}})$ \\
\textbf{Auxiliary sheaf} 
& $\mathcal{Q}$ 
& $\mathcal{H}^x_{\mathcal{I}}$ \\
\textbf{Laplacian} 
& $L_{\mathcal{Q}}$ 
& $L_{\mathcal{H}^x_{\mathcal{I}}}$ \\
\textbf{Forcing term} 
& $L_{QS} u$ 
& $A^T c$ \\
\textbf{Poisson equation} 
& $L_{\mathcal{Q}} y = -L_{QS} u$ 
& $L_{\mathcal{H}^x_{\mathcal{I}}} \rho = -A^T c$ \\
\bottomrule
\end{tabular}
\end{table}


\section{Constrained Joint Diffusion with Stubborn Agents}\label{sec:joint}

Sections~\ref{sec:stubborn-opinions} and~\ref{sec:stubborn-expressions} developed parallel theories: stubborn opinions induce constrained dynamics on a subsheaf, and stubborn expressions induce constrained dynamics on an auxiliary sheaf of structures. In practice, beliefs and communication evolve together, and their interaction determines whether a network reaches consensus, settles into productive disagreement, or fragments. We now unify these perspectives by allowing agents to be stubborn in their private beliefs while the network's communication structure evolves subject to its own constraints.

The joint dynamics depend critically on which communication channels are permitted to adapt. We identify four natural scenarios distinguished by the adaptation patterns on edges incident to stubborn vertices. In Scenarios 1 and 2, all edges have symmetric adaptation status (both endpoints adapt or both freeze); in Scenarios 3 and 4, mixed edges exhibit asymmetric adaptation, creating qualitatively different dynamics.

This section proceeds as follows. We first set up the constrained joint dynamics (\S\ref{sec:joint-dynamics}), then classify edges by their adaptation symmetry (\S\ref{ssec:edge-classification}). The four scenarios are described in \S\ref{ssec:four-scenarios}, and convergence analysis follows in \S\ref{ssec:convergence}. We conclude with conservation laws that ensure the system cannot reach trivial equilibria through opinion collapse or communication shutdown (\S\ref{ssec:conservation-laws}).


\subsection{The Constrained Joint Dynamics}\label{sec:joint-dynamics}

We retain the notation of Sections~\ref{sec:stubborn-opinions} and~\ref{sec:stubborn-expressions}. For opinions: $U \subseteq V$ is the set of stubborn vertices, $\mathcal{F}(v) = S_v \oplus T_v$ decomposes each stalk into stubborn and free directions, and the global state is $x = \tilde{u} + \tilde{y}$ where $\tilde{u} = \iota_S(u)$ is fixed and $\tilde{y} = \iota_{\mathcal{Q}}(y)$ evolves. For structures: $\mathcal{I} \subseteq \{(v,e) : v \unlhd e\}$ is the set of adapting incidences, and the coboundary operator $\delta$ encodes the full collection of restriction maps $\mathcal{F}_{v \unlhd e}$. We write $F = V \setminus U$ for the set of free vertices.

To formulate the constrained dynamics, let
\[
\mathcal{M} = \left\{\frac{d\delta}{dt} : \frac{d}{dt}(\delta_{v \unlhd e}) = 0 \text{ for all } (v,e) \notin \mathcal{I}\right\}
\]
be the subspace of structure velocities that leave non-adapting restriction maps unchanged, and let $\Pi_{\mathcal{M}}$ denote orthogonal projection onto $\mathcal{M}$ with respect to the Frobenius inner product. Note that $\mathcal{M}$ is isomorphic to $V_{\mathrm{maps}}$ from Section~\ref{sec:stubborn-expressions}: specifying a velocity in $\mathcal{M}$ is equivalent to specifying how each adapting restriction map changes.

The \emph{constrained joint dynamics} couple opinion evolution with structure adaptation:
\begin{equation}\label{eq:constrained_joint}
\begin{aligned}
\frac{dy}{dt} &= -\alpha \, P_{\mathcal{Q}} \bigl( \delta(t)^T \delta(t) \, x(t) \bigr), \\
\frac{d\delta}{dt} &= -\beta \, \Pi_{\mathcal{M}} \bigl( \delta(t) \, x(t) \, x(t)^T \bigr).
\end{aligned}
\end{equation}
The first equation is sheaf diffusion projected onto the free opinion space, as in Section~\ref{sec:stubborn-opinions}. The second is gradient descent on the discrepancy energy $\frac{1}{2}\|\delta x\|^2$ with respect to the restriction maps, projected onto allowable velocities. The rates $\alpha, \beta > 0$ govern the relative speeds of opinion and structure adaptation. Unconstrained joint dynamics of this form have been studied for cellular sheaves of real vector spaces by Hansen and Ghrist~\cite{hansen2020}, for cellular sheaves of Hilbert spaces by Gould~\cite{gould2024hilbert}, and discretized into sheaf neural network architectures by Hernandez Caralt et al.~\cite{hernandezcaralt2024joint}.

The right-hand side of \eqref{eq:constrained_joint} is polynomial in $(y, \delta)$, so standard ODE theory guarantees local existence and uniqueness of solutions. Global existence holds on any trajectory that remains bounded. When $E_A = \varnothing$, boundedness of $\|\delta\|_F$ follows from Theorem~\ref{thm:constrained_convergence}; boundedness of $y$ is ensured by either $H^0(G; \mathcal{Q}) = 0$ or regularization (Appendix~\ref{app:regularization}).


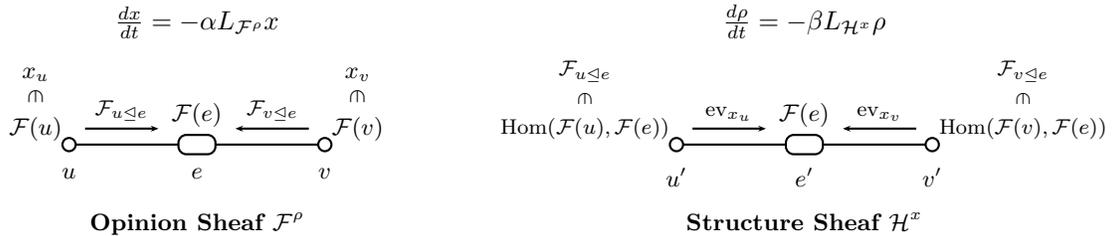
\begin{figure}[ht]
\centering
\begin{tikzpicture}[
    vertex/.style={circle, fill=white, draw=black, thick, minimum size=5pt, inner sep=0pt},
    edge stalk h/.style={rectangle, rounded corners=3pt, fill=white, draw=black, thick, minimum width=14pt, minimum height=4pt},
    arrow/.style={-{Stealth[length=3pt]}, thick},
    every node/.style={font=\footnotesize},
    scale=0.85
]


\node[above, font=\small] at (2,1.5) {$\frac{dx}{dt} = -\alpha L_{\mathcal{F}^{\rho}} x$};

\draw[thick] (0,0) -- (4,0);

\node[edge stalk h] (e1) at (2,0) {};

\node[vertex] (u1) at (0,0) {};
\node[vertex] (v1) at (4,0) {};

\node[below=5pt] at (u1) {$u$};
\node[below=5pt] at (v1) {$v$};
\node[below=5pt] at (e1) {$e$};

\node[above=6pt, anchor=east] (Fu) at (0.05,0) {$\mathcal{F}(u)$};
\node[above=6pt, anchor=west] (Fv) at (3.95,0) {$\mathcal{F}(v)$};

\node[above=3pt] at (e1) {$\mathcal{F}(e)$};

\node[above=-3pt] at (Fu.north) {\rotatebox{-90}{$\in$}};
\node[above=6pt] at (Fu.north) {$x_u$};

\node[above=-3pt] at (Fv.north) {\rotatebox{-90}{$\in$}};
\node[above=6pt] at (Fv.north) {$x_v$};

\draw[arrow] (0.25,0.25) -- (1.4,0.25);
\node[above=-1pt, font=\scriptsize] at (0.82,0.25) {$\mathcal{F}_{u \unlhd e}$};

\draw[arrow] (3.75,0.25) -- (2.6,0.25);
\node[above=-1pt, font=\scriptsize] at (3.18,0.25) {$\mathcal{F}_{v \unlhd e}$};

\node[below=22pt] at (2,0) {\textbf{Opinion Sheaf $\mathcal{F}^{\rho}$}};


\begin{scope}[xshift=9.5cm]

\node[above, font=\small] at (2,1.5) {$\frac{d\rho}{dt}  = -\beta L_{\mathcal{H}^{x}} \rho$};

\draw[thick] (0,0) -- (4,0);

\node[edge stalk h] (e2) at (2,0) {};

\node[vertex] (u2) at (0,0) {};
\node[vertex] (v2) at (4,0) {};

\node[below=5pt] at (u2) {$u'$};
\node[below=5pt] at (v2) {$v'$};
\node[below=5pt] at (e2) {$e'$};

\node[above=6pt, anchor=east, font=\scriptsize] (Hu) at (0.05,0) {$\mathrm{Hom}(\mathcal{F}(u), \mathcal{F}(e))$};
\node[above=6pt, anchor=west, font=\scriptsize] (Hv) at (3.95,0) {$\mathrm{Hom}(\mathcal{F}(v), \mathcal{F}(e))$};

\node[above=3pt] at (e2) {$\mathcal{F}(e)$};

\node[above=-3pt, font=\scriptsize] at (Hu.north) {\rotatebox{-90}{$\in$}};
\node[above=6pt, font=\scriptsize] at (Hu.north) {$\mathcal{F}_{u \unlhd e}$};

\node[above=-3pt, font=\scriptsize] at (Hv.north) {\rotatebox{-90}{$\in$}};
\node[above=6pt, font=\scriptsize] at (Hv.north) {$\mathcal{F}_{v \unlhd e}$};

\draw[arrow] (0.25,0.25) -- (1.4,0.25);
\node[above=-1pt, font=\scriptsize] at (0.82,0.25) {$\mathrm{ev}_{x_u}$};

\draw[arrow] (3.75,0.25) -- (2.6,0.25);
\node[above=-1pt, font=\scriptsize] at (3.18,0.25) {$\mathrm{ev}_{x_v}$};

\node[below=22pt] at (2,0) {\textbf{Structure Sheaf $\mathcal{H}^{x}$}};

\end{scope}

\end{tikzpicture}
\caption{The coupled dynamics of opinions and communication structures. \textbf{Left:} The opinion sheaf $\mathcal{F}^{\rho}$ over a single edge, with opinions $x_u \in \mathcal{F}(u)$ and $x_v \in \mathcal{F}(v)$ evolving under the sheaf Laplacian $L_{\mathcal{F}^{\rho}}$. \textbf{Right:} The structure sheaf $\mathcal{H}^{x}$ whose vertex stalks are the spaces $\mathrm{Hom}(\mathcal{F}(u), \mathcal{F}(e))$ and $\mathrm{Hom}(\mathcal{F}(v), \mathcal{F}(e))$; when stalks are finite-dimensional real vector spaces, these are simply spaces of matrices of compatible dimensions. The restriction maps $\mathrm{ev}_{x_u}$ and $\mathrm{ev}_{x_v}$ are evaluation maps that apply each matrix to the current opinion vector, i.e., right multiplication $\cdot \, x_u$ and  $\cdot \, x_v$. The two dynamics are coupled: $L_{\mathcal{F}^{\rho}}$ depends on the current structure $\rho$, while $L_{\mathcal{H}^{x}}$ depends on the current opinions $x$.}
\label{fig:coupled-sheaves}
\end{figure}

\subsection{Classification by Edge Type}\label{ssec:edge-classification}

The behavior of the system depends on two factors: which vertices hold stubborn opinions, and which restriction maps are permitted to adapt. We introduce two complementary partitions of the edge set.

The first partition is based on the stubbornness of incident vertices:
\begin{align*}
E_{FF} &= \{e \in E : \text{both endpoints in } F\}, \\
E_{UU} &= \{e \in E : \text{both endpoints in } U\}, \\
E_{UF} &= \{e \in E : \text{one endpoint in each}\}.
\end{align*}
Edges in $E_{FF}$ connect vertices with no stubborn directions ($S_v = \{0\}$ at both endpoints), so both agents' entire opinion vectors evolve. Edges in $E_{UU}$ connect vertices that each have at least one stubborn direction 
($S_v \neq \{0\}$ at both endpoints), though these agents may still have free subspaces. Edges in $E_{UF}$ are mixed, connecting a vertex with stubborn directions to one without.

The second classification concerns the adaptation pattern of restriction maps.

\begin{definition}[Edge Adaptation Type]
An edge $e = \{u, v\}$ is \emph{Type S} (symmetric) if both incidences share the same adaptation status: either $(u, e), (v, e) \in \mathcal{I}$ or $(u, e), (v, e) \notin \mathcal{I}$. An edge is \emph{Type A} (asymmetric) if exactly one of $(u,e), (v,e)$ belongs to $\mathcal{I}$. We write $E_S$ and $E_A$ for the sets of Type S and Type A edges.
\end{definition}

This classification determines key analytical properties. As we establish in Theorem~\ref{thm:constrained_convergence}, the Frobenius norm $\|\delta\|_F$ is nonincreasing along trajectories if and only if $E_A = \varnothing$. When Type A edges are present, asymmetric adaptation can cause the norm to increase, and regularization may be required to ensure bounded trajectories.

\subsection{The Four Scenarios}\label{ssec:four-scenarios}

The edge partitions $E_{FF}$, $E_{UU}$, $E_{UF}$ and the adaptation set $\mathcal{I}$ can be combined in many ways. We identify four representative \emph{edge policies} that arise naturally in social contexts, illustrated in Figure~\ref{fig:four-scenarios} and summarized in Table~\ref{tab:edge-policies}. Rather than global network configurations, these scenarios describe local rules for individual edges; a given network may apply different policies to different edges.

Scenarios 1 and 2 apply to edges of any type.

\textbf{Scenario 1: Universal Adaptation.} Both restriction maps evolve, regardless of whether the endpoints are free or stubborn. Agents may hold firm beliefs in certain directions, but they remain diplomatically flexible in how they express themselves to all neighbors. A scientist with strong methodological commitments might still learn to explain their work differently to collaborators, funding agencies, and the public.

\textbf{Scenario 2: Structural Stubbornness.} Both restriction maps are frozen. Communication on this edge follows fixed protocols regardless of how opinions evolve. This models institutional rigidity or longstanding conventions that persist even as underlying beliefs shift.

Scenarios 3 and 4 apply specifically to mixed edges in $E_{UF}$, introducing asymmetric adaptation.

\textbf{Scenario 3: Accommodation.} The free agent adapts toward the stubborn neighbor while the stubborn agent maintains fixed expression. This models deference: newcomers learn to speak the language of established members, employees adapt to management's communication style, students learn disciplinary conventions from advisors.

\textbf{Scenario 4: Outreach.} The stubborn agent adapts expression to reach the free neighbor while the free agent communicates authentically. This models persuasion: a politician tailors rhetoric to different constituencies, a teacher adjusts explanations to student backgrounds, an activist frames arguments for skeptical audiences.

A network may apply different policies to different edges. For instance, one might use Scenario~1 on $E_{FF}$ edges, Scenario~2 on $E_{UU}$ edges, and Scenario~3 on $E_{UF}$ edges. The key analytical distinction is whether each edge is Type~S or Type~A. Scenarios~1 and~2 produce Type~S edges (both incidences share the same adaptation status). Scenarios~3 and~4 produce Type~A edges (exactly one incidence adapts). Theorem~\ref{thm:constrained_convergence} shows that $\|\delta\|_F$ is monotone nonincreasing, hence a priori bounded, exactly when all edges are Type~S. In the presence of Type~A edges, $\|\delta\|_F$ may increase, and boundedness requires additional hypotheses or regularization.

Scenarios~3 and~4 superficially resemble antagonistic networks~\cite{altafini2012,altafini2013,altafiniCeragioli2018,altafiniLini2015}, where negative edges encode distrust and agents seek to disagree. The key difference is that here agents still seek agreement, but adaptation is asymmetric. When negative restriction maps do emerge, they reflect learned accommodation rather than structural antagonism~\cite{antal2005,antal2006}.

\begin{table}[ht]
\centering
\caption{Edge policies and their adaptation patterns. For Scenarios~3--4 (mixed edges), $u$ denotes the stubborn endpoint and $v$ the free endpoint.}
\label{tab:edge-policies}
\begin{tabular}{@{}clccc@{}}
\toprule
Scenario & Policy & $(u,e) \in \mathcal{I}$? & $(v,e) \in \mathcal{I}$? & Type \\
\midrule
1 & Universal Adaptation & Yes & Yes & S \\
2 & Structural Stubbornness & No & No & S \\
3 & Accommodation & No & Yes & A \\
4 & Outreach & Yes & No & A \\
\bottomrule
\end{tabular}
\end{table}


\begin{figure}[ht]
\centering
\begin{tikzpicture}[
    vertex/.style={circle, fill=white, draw=black, thick, minimum size=5pt, inner sep=0pt},
    edge stalk h/.style={rectangle, rounded corners=3pt, fill=white, draw=black, thick, minimum width=16pt, minimum height=4pt},
    arrow/.style={-{Stealth[length=4pt]}, thick},
    every node/.style={font=\small},
    scale=1.0
]


\draw[thick] (0,0) -- (5,0);

\node[edge stalk h] (e1) at (2.5,0) {};
\node[below=6pt] at (e1) {$e$};

\node[vertex] (u1) at (0,0) {};
\node[vertex] (v1) at (5,0) {};

\node[below=6pt] at (u1) {$u$};
\node[below=6pt] at (v1) {$v$};

\node[above left] at (u1) {${\color{red}\mathbb{R}} \oplus \mathbb{R}$};
\node[above right] at (v1) {$\mathbb{R}$};

\node[above=3pt] at (e1) {$\mathbb{R}$};

\node[above=18pt, font=\footnotesize] at (-0.5,0) {$x_u = \begin{bmatrix} {\color{red}1} \\ 1 \end{bmatrix}$};
\node[above=18pt, font=\footnotesize] at (5.5,0) {$x_v = -1$};

\draw[arrow] (0.3,0.3) -- (1.8,0.3);
\node[above, font=\scriptsize] at (1.05,0.3) {$\begin{bmatrix} \frac{1}{2} & \frac{1}{2} \end{bmatrix}$};

\draw[arrow] (4.7,0.3) -- (3.2,0.3);
\node[above, font=\scriptsize] at (3.95,0.3) {$[1]$};

\node[below=28pt, font=\footnotesize] at (2.5,0) {\textbf{Scenario 1: Universal Adaptation}};


\begin{scope}[xshift=9cm]

\draw[thick] (0,0) -- (5,0);

\node[edge stalk h] (e2) at (2.5,0) {};
\node[below=6pt] at (e2) {$e$};

\node[vertex] (u2) at (0,0) {};
\node[vertex] (v2) at (5,0) {};

\node[below=6pt] at (u2) {$u$};
\node[below=6pt] at (v2) {$v$};

\node[above left] at (u2) {${\color{red}\mathbb{R}} \oplus \mathbb{R}$};
\node[above right] at (v2) {$\mathbb{R}$};

\node[above=3pt] at (e2) {$\mathbb{R}$};

\node[above=18pt, font=\footnotesize] at (-0.5,0) {$x_u = \begin{bmatrix} {\color{red}1} \\ 1 \end{bmatrix}$};
\node[above=18pt, font=\footnotesize] at (5.5,0) {$x_v = -1$};

\draw[arrow, red] (0.3,0.3) -- (1.8,0.3);
\node[above, font=\scriptsize, red] at (1.05,0.3) {$\begin{bmatrix} \frac{1}{2} & \frac{1}{2} \end{bmatrix}$};

\draw[arrow, red] (4.7,0.3) -- (3.2,0.3);
\node[above, font=\scriptsize, red] at (3.95,0.3) {$[1]$};

\node[below=28pt, font=\footnotesize] at (2.5,0) {\textbf{Scenario 2: Structural Stubbornness}};

\end{scope}


\begin{scope}[yshift=-4cm]

\draw[thick] (0,0) -- (5,0);

\node[edge stalk h] (e3) at (2.5,0) {};
\node[below=6pt] at (e3) {$e$};

\node[vertex] (u3) at (0,0) {};
\node[vertex] (v3) at (5,0) {};

\node[below=6pt] at (u3) {$u$};
\node[below=6pt] at (v3) {$v$};

\node[above left] at (u3) {${\color{red}\mathbb{R}} \oplus \mathbb{R}$};
\node[above right] at (v3) {$\mathbb{R}$};

\node[above=3pt] at (e3) {$\mathbb{R}$};

\node[above=18pt, font=\footnotesize] at (-0.5,0) {$x_u = \begin{bmatrix} {\color{red}1} \\ 1 \end{bmatrix}$};
\node[above=18pt, font=\footnotesize] at (5.5,0) {$x_v = -1$};

\draw[arrow, red] (0.3,0.3) -- (1.8,0.3);
\node[above, font=\scriptsize, red] at (1.05,0.3) {$\begin{bmatrix} \frac{1}{2} & \frac{1}{2} \end{bmatrix}$};

\draw[arrow] (4.7,0.3) -- (3.2,0.3);
\node[above, font=\scriptsize] at (3.95,0.3) {$[1]$};

\node[below=28pt, font=\footnotesize] at (2.5,0) {\textbf{Scenario 3: Accommodation}};

\end{scope}


\begin{scope}[xshift=9cm, yshift=-4cm]

\draw[thick] (0,0) -- (5,0);

\node[edge stalk h] (e4) at (2.5,0) {};
\node[below=6pt] at (e4) {$e$};

\node[vertex] (u4) at (0,0) {};
\node[vertex] (v4) at (5,0) {};

\node[below=6pt] at (u4) {$u$};
\node[below=6pt] at (v4) {$v$};

\node[above left] at (u4) {${\color{red}\mathbb{R}} \oplus \mathbb{R}$};
\node[above right] at (v4) {$\mathbb{R}$};

\node[above=3pt] at (e4) {$\mathbb{R}$};

\node[above=18pt, font=\footnotesize] at (-0.5,0) {$x_u = \begin{bmatrix} {\color{red}1} \\ 1 \end{bmatrix}$};
\node[above=18pt, font=\footnotesize] at (5.5,0) {$x_v = -1$};

\draw[arrow] (0.3,0.3) -- (1.8,0.3);
\node[above, font=\scriptsize] at (1.05,0.3) {$\begin{bmatrix} \frac{1}{2} & \frac{1}{2} \end{bmatrix}$};

\draw[arrow, red] (4.7,0.3) -- (3.2,0.3);
\node[above, font=\scriptsize, red] at (3.95,0.3) {$[1]$};

\node[below=28pt, font=\footnotesize] at (2.5,0) {\textbf{Scenario 4: Outreach}};

\end{scope}

\end{tikzpicture}
\caption{The four edge policies illustrated on a mixed edge ($E_{UF}$). Agent $u$ has opinion space $\mathbb{R} \oplus \mathbb{R}$ with the first coordinate stubborn (shown in red); agent $v$ is free with opinion space $\mathbb{R}$. Restriction maps are colored red when frozen and black when adapting. \textbf{Scenario 1 (Universal Adaptation):} Both maps evolve. \textbf{Scenario 2 (Structural Stubbornness):} Both maps are frozen. \textbf{Scenario 3 (Accommodation):} The free agent adapts; the stubborn agent's map is frozen. \textbf{Scenario 4 (Outreach):} The stubborn agent adapts; the free agent's map is frozen. Scenarios~1 and~2 apply to any edge; Scenarios~3 and~4 apply only to mixed edges.}
\label{fig:four-scenarios}
\end{figure}

\subsection{Convergence Analysis}\label{ssec:convergence}

The joint dynamics~\eqref{eq:constrained_joint} are gradient flow of a nonconvex energy on the affine subspace defined by frozen coordinates (stubborn opinions and non-adapting restriction maps). The central object is the total disagreement energy
\[
\Psi(y, \delta) = \tfrac{1}{2}\|\delta x\|^2,
\]
which measures how far the system is from perfect expressed consensus. We show that $\Psi$ decreases along trajectories, but boundedness of the state $(y, \delta)$ depends on the edge adaptation structure.

\begin{theorem}[Convergence of Constrained Joint Diffusion]\label{thm:constrained_convergence}
Let $(y(t), \delta(t))$ be a solution to~\eqref{eq:constrained_joint} with initial conditions $(y_0, \delta_0)$. The total disagreement energy $\Psi(y, \delta) = \frac{1}{2}\|\delta x\|^2$ is nonincreasing along trajectories. Moreover, the Frobenius norm $\|\delta\|_F$ is nonincreasing if and only if every edge is Type~S (equivalently, $E_A = \varnothing$).

If the forward orbit $\{(y(t), \delta(t)) : t \geq 0\}$ is precompact, then the trajectory converges to the set of equilibria.
\end{theorem}

When $E_A = \varnothing$, boundedness of $\|\delta\|_F$ is guaranteed. Boundedness of $y$ follows if $H^0(G; \mathcal{Q}) = 0$; otherwise the harmonic component may drift and precompactness requires regularization. When $E_A \neq \varnothing$, the Frobenius norm $\|\delta\|_F$ may increase along trajectories; Appendix~\ref{app:regularization} provides an explicit example and discusses regularization.

\begin{proof}
We establish the two claims in turn. For energy dissipation, since $x = \tilde{u} + \tilde{y}$ with $\tilde{u}$ fixed, we have $\frac{dx}{dt} = \iota_{\mathcal{Q}}(\frac{dy}{dt})$. The time derivative of $\Psi = \frac{1}{2}\|\delta x\|^2$ is
\[
\frac{d\Psi}{dt} = \langle \delta x, \frac{d\delta}{dt} x \rangle + \langle \delta^T \delta x, \frac{dx}{dt} \rangle.
\]
For the first term, observe that $\langle \delta x, \frac{d\delta}{dt} x \rangle = \langle \frac{d\delta}{dt}, \delta x\, x^T \rangle_F$. Substituting the dynamics~\eqref{eq:constrained_joint}:
\[
\frac{d\Psi}{dt} = -\beta \langle \Pi_{\mathcal{M}}(\delta x\, x^T), \delta x\, x^T \rangle_F - \alpha \langle \delta^T \delta x, \iota_{\mathcal{Q}} P_{\mathcal{Q}}(\delta^T \delta x) \rangle.
\]
Since $\Pi_{\mathcal{M}}$ and $\iota_{\mathcal{Q}} P_{\mathcal{Q}}$ are orthogonal projections, we have $\langle \Pi_{\mathcal{M}}(A), A \rangle_F = \|\Pi_{\mathcal{M}}(A)\|_F^2$ and $\langle v, \iota_{\mathcal{Q}} P_{\mathcal{Q}}(v) \rangle = \|P_{\mathcal{Q}}(v)\|^2$. Therefore
\[
\frac{d\Psi}{dt} = -\beta \|\Pi_{\mathcal{M}}(\delta x\, x^T)\|_F^2 - \alpha \|P_{\mathcal{Q}}(\delta^T \delta x)\|^2 \leq 0.
\]

For the second claim, we analyze the derivative of $\|\delta\|_F^2$ edge by edge. Fix an oriented edge $e = u \to v$, write $\rho_u = \mathcal{F}_{u \unlhd e}$ and $\rho_v = \mathcal{F}_{v \unlhd e}$, and let $a = \rho_v(x_v)$ and $b = \rho_u(x_u)$, so that $(\delta x)_e = a - b$. The gradient of $\Psi$ with respect to the restriction maps at edge $e$ is
\[
\nabla_{\rho_u} \Psi = -(\delta x)_e x_u^T, \qquad \nabla_{\rho_v} \Psi = (\delta x)_e x_v^T.
\]
The dynamics $\frac{d\delta}{dt} = -\beta \Pi_{\mathcal{M}}(\nabla_\delta \Psi)$ give, for adapting incidences,
\[
\frac{d\rho_u}{dt} = \beta(\delta x)_e x_u^T \quad \text{if } (u,e) \in \mathcal{I}, \qquad
\frac{d\rho_v}{dt} = -\beta(\delta x)_e x_v^T \quad \text{if } (v,e) \in \mathcal{I},
\]
with $d\rho_w/dt = 0$ for $(w,e) \notin \mathcal{I}$. The contribution of edge $e$ to $\frac{d}{dt}\|\delta\|_F^2 = 2\langle \delta, \frac{d\delta}{dt} \rangle_F$ is
\[
2\langle \rho_u, \frac{d\rho_u}{dt} \rangle_F + 2\langle \rho_v, \frac{d\rho_v}{dt} \rangle_F 
= 2\beta \mathbf{1}_{(u,e) \in \mathcal{I}} \langle b, a-b \rangle - 2\beta \mathbf{1}_{(v,e) \in \mathcal{I}} \langle a, a-b \rangle.
\]

When both incidences adapt (Type~S), this equals $2\beta \langle b - a, a - b \rangle = -2\beta \|(\delta x)_e\|^2 \leq 0$. When both are frozen (also Type~S), the contribution is zero. When exactly one incidence adapts (Type~A), a short calculation shows that the contribution equals $2\beta(\langle a, b \rangle - \|b\|^2)$ if only $u$ adapts, or $2\beta(\langle a, b \rangle - \|a\|^2)$ if only $v$ adapts. Either quantity can be positive for suitable initial conditions.

Summing over all edges: if $E_A = \varnothing$, every contribution is nonpositive, so $\frac{d}{dt}\|\delta\|_F^2 \leq 0$. Conversely, if $E_A \neq \varnothing$, there exist initial conditions for which some Type~A edge contributes positively, and when this exceeds the nonpositive contributions from Type~S edges, $\frac{d}{dt}\|\delta\|_F^2 > 0$. Convergence to the equilibrium set under precompactness follows from LaSalle's invariance principle.
\end{proof}

\begin{remark}[Equilibrium structure]
At equilibrium, the condition $\frac{d\delta}{dt} = 0$ requires $(\delta x)_e x_v^T = 0$ for each adapting incidence $(v,e) \in \mathcal{I}$. When $x_v \neq 0$, this forces $(\delta x)_e = 0$. When $x_v = 0$, the condition holds vacuously and the restriction map $\mathcal{F}_{v \unlhd e}$ remains fixed, consistent with the zero restriction maps in $\mathcal{H}^x$ at such vertices. Edges governed by Scenario~2 (both incidences frozen) impose no constraint on $(\delta x)_e$ and may retain residual discrepancy. Edges governed by Scenarios~1, 3, or~4 each have at least one adapting incidence, so $(\delta x)_e = 0$ is achieved whenever the adapting endpoint has nonzero opinion.
\end{remark}


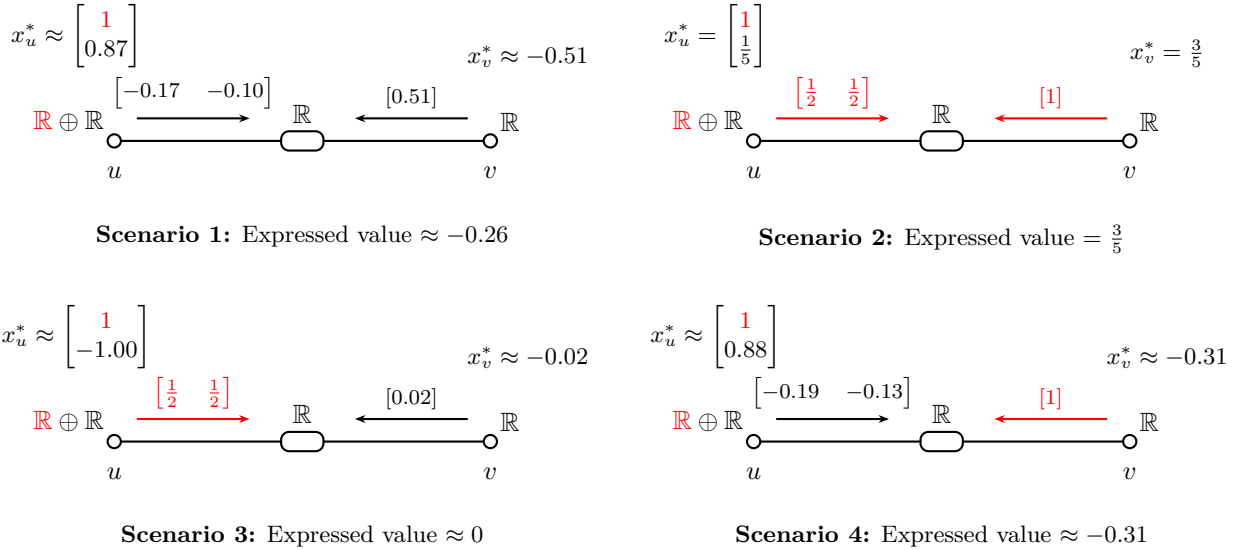
\begin{figure}[ht]
\centering
\begin{tikzpicture}[
    vertex/.style={circle, fill=white, draw=black, thick, minimum size=5pt, inner sep=0pt},
    edge stalk h/.style={rectangle, rounded corners=3pt, fill=white, draw=black, thick, minimum width=16pt, minimum height=4pt},
    arrow/.style={-{Stealth[length=4pt]}, thick},
    every node/.style={font=\small},
    scale=1.0
]


\draw[thick] (0,0) -- (5,0);
\node[edge stalk h] at (2.5,0) {};
\node[vertex] at (0,0) {};
\node[vertex] at (5,0) {};
\node[below=6pt] at (0,0) {$u$};
\node[below=6pt] at (5,0) {$v$};
\node[above left] at (0,0) {${\color{red}\mathbb{R}} \oplus \mathbb{R}$};
\node[above right] at (5,0) {$\mathbb{R}$};
\node[above=3pt] at (2.5,0) {$\mathbb{R}$};
\node[above=24pt, font=\footnotesize] at (-0.5,0) {$x_u^* \approx \begin{bmatrix} {\color{red}1} \\ 0.87 \end{bmatrix}$};
\node[above=24pt, font=\footnotesize] at (5.5,0) {$x_v^* \approx -0.51$};
\draw[arrow] (0.3,0.3) -- (1.8,0.3);
\node[above, font=\scriptsize] at (1.05,0.3) {$\begin{bmatrix} -0.17 & -0.10 \end{bmatrix}$};
\draw[arrow] (4.7,0.3) -- (3.2,0.3);
\node[above, font=\scriptsize] at (3.95,0.3) {$[0.51]$};
\node[below=28pt, font=\footnotesize] at (2.5,0) {\textbf{Scenario 1:} Expressed value $\approx -0.26$};


\begin{scope}[xshift=8.5cm]
\draw[thick] (0,0) -- (5,0);
\node[edge stalk h] at (2.5,0) {};
\node[vertex] at (0,0) {};
\node[vertex] at (5,0) {};
\node[below=6pt] at (0,0) {$u$};
\node[below=6pt] at (5,0) {$v$};
\node[above left] at (0,0) {${\color{red}\mathbb{R}} \oplus \mathbb{R}$};
\node[above right] at (5,0) {$\mathbb{R}$};
\node[above=3pt] at (2.5,0) {$\mathbb{R}$};
\node[above=24pt, font=\footnotesize] at (-0.5,0) {$x_u^* = \begin{bmatrix} {\color{red}1} \\ \frac{1}{5} \end{bmatrix}$};
\node[above=24pt, font=\footnotesize] at (5.5,0) {$x_v^* = \frac{3}{5}$};
\draw[arrow, red] (0.3,0.3) -- (1.8,0.3);
\node[above, font=\scriptsize, red] at (1.05,0.3) {$\begin{bmatrix} \frac{1}{2} & \frac{1}{2} \end{bmatrix}$};
\draw[arrow, red] (4.7,0.3) -- (3.2,0.3);
\node[above, font=\scriptsize, red] at (3.95,0.3) {$[1]$};
\node[below=28pt, font=\footnotesize] at (2.5,0) {\textbf{Scenario 2:} Expressed value $= \frac{3}{5}$};
\end{scope}


\begin{scope}[yshift=-4cm]
\draw[thick] (0,0) -- (5,0);
\node[edge stalk h] at (2.5,0) {};
\node[vertex] at (0,0) {};
\node[vertex] at (5,0) {};
\node[below=6pt] at (0,0) {$u$};
\node[below=6pt] at (5,0) {$v$};
\node[above left] at (0,0) {${\color{red}\mathbb{R}} \oplus \mathbb{R}$};
\node[above right] at (5,0) {$\mathbb{R}$};
\node[above=3pt] at (2.5,0) {$\mathbb{R}$};
\node[above=24pt, font=\footnotesize] at (-0.5,0) {$x_u^* \approx \begin{bmatrix} {\color{red}1} \\ -1.00 \end{bmatrix}$};
\node[above=24pt, font=\footnotesize] at (5.5,0) {$x_v^* \approx -0.02$};
\draw[arrow, red] (0.3,0.3) -- (1.8,0.3);
\node[above, font=\scriptsize, red] at (1.05,0.3) {$\begin{bmatrix} \frac{1}{2} & \frac{1}{2} \end{bmatrix}$};
\draw[arrow] (4.7,0.3) -- (3.2,0.3);
\node[above, font=\scriptsize] at (3.95,0.3) {$[0.02]$};
\node[below=28pt, font=\footnotesize] at (2.5,0) {\textbf{Scenario 3:} Expressed value $\approx 0$};
\end{scope}


\begin{scope}[xshift=8.5cm, yshift=-4cm]
\draw[thick] (0,0) -- (5,0);
\node[edge stalk h] at (2.5,0) {};
\node[vertex] at (0,0) {};
\node[vertex] at (5,0) {};
\node[below=6pt] at (0,0) {$u$};
\node[below=6pt] at (5,0) {$v$};
\node[above left] at (0,0) {${\color{red}\mathbb{R}} \oplus \mathbb{R}$};
\node[above right] at (5,0) {$\mathbb{R}$};
\node[above=3pt] at (2.5,0) {$\mathbb{R}$};
\node[above=24pt, font=\footnotesize] at (-0.5,0) {$x_u^* \approx \begin{bmatrix} {\color{red}1} \\ 0.88 \end{bmatrix}$};
\node[above=24pt, font=\footnotesize] at (5.5,0) {$x_v^* \approx -0.31$};
\draw[arrow] (0.3,0.3) -- (1.8,0.3);
\node[above, font=\scriptsize] at (1.05,0.3) {$\begin{bmatrix} -0.19 & -0.13 \end{bmatrix}$};
\draw[arrow, red] (4.7,0.3) -- (3.2,0.3);
\node[above, font=\scriptsize, red] at (3.95,0.3) {$[1]$};
\node[below=28pt, font=\footnotesize] at (2.5,0) {\textbf{Scenario 4:} Expressed value $\approx -0.31$};
\end{scope}

\end{tikzpicture}
\caption{Limit sheaves for the four scenarios, starting from the symmetric initial condition of Figure~\ref{fig:four-scenarios} where expressed opinions are $+1$ from $u$ and $-1$ from $v$. In \textbf{Scenario 2}, only opinions adjust, reaching the compromise $\frac{3}{5}$. In \textbf{Scenario 3}, agent $v$ silences itself, achieving consensus near zero. In \textbf{Scenarios 1 and 4}, the restriction maps become negative, meaning agents invert how they translate beliefs into expressions. See Appendix~\ref{app:example-calculations} for derivations.}
\label{fig:four-scenarios-limit}
\end{figure}

\subsection{Conservation Laws and Nontrivial Equilibria}\label{ssec:conservation-laws}

A natural concern is whether the system might achieve agreement by degenerate means: all opinions vanishing ($x \to 0$) or all communication ceasing ($\delta \to 0$). Hansen and Ghrist~\cite[Theorem~9.3]{hansen2020} showed that in the unconstrained case a conservation law obstructs such collapse. We extend their analysis to the constrained setting and show that the same conserved quantity governs both opinion persistence and discourse persistence.

For each vertex $v$, define the time-dependent matrix
\[
Q_{vv}(t) = \alpha (L_{\mathcal{F}}(t))_{vv} - \beta x_v(t) x_v(t)^T \in \mathrm{Hom}(\mathcal{F}(v), \mathcal{F}(v)).
\]
This symmetric matrix balances communication strength against opinion magnitude at vertex $v$. The diagonal block $(L_{\mathcal{F}})_{vv}$ aggregates the restriction maps on edges incident to $v$, while the rank-one term $x_v x_v^T$ captures the direction and intensity of $v$'s opinion.

\begin{proposition}[Persistence of Nontrivial Dynamics]\label{prop:persistence}
Consider the constrained joint dynamics~\eqref{eq:constrained_joint} with precompact forward orbit. If $v$ is an interior free vertex and all of its incidences adapt  (formally $v$ has all incident edges in $E_{FF}$ and $(v,e)$ in $\mathcal{I}$), then $Q_{vv}(t)$ is constant in time. Under this conservation condition, a negative eigenvalue of $Q_{vv}(0)$ implies $x_v(t) \not\to 0$, while a positive eigenvalue implies that the restriction maps incident to $v$ cannot all vanish.
\end{proposition}

\begin{proof}
The conservation follows from the computation in~\cite[Theorem~9.3]{hansen2020}. When all restriction maps originating from $v$ evolve according to the gradient flow, the contributions to $d{Q}_{vv}/dt$ from opinion evolution and from map evolution cancel exactly. Free vertices with all incident edges in $E_{FF}$ and all incidences adapting satisfy this condition.

For opinion persistence, suppose toward contradiction that $x_v(t) \to 0$. By precompactness, along a convergent subsequence we have $\delta(t_{n_k}) \to \delta_\infty$. In the limit, $Q_{vv}(0) = \alpha (L_{\mathcal{F}}(\infty))_{vv}$, which is positive semidefinite. This contradicts the existence of a negative eigenvalue.

For discourse persistence, suppose $\delta(t) \to 0$. Again by precompactness, $x(t_{n_k}) \to x_\infty$ along a subsequence. In the limit, $Q_{vv}(0) = -\beta (x_\infty)_v (x_\infty)_v^T$, which is negative semidefinite. This contradicts the existence of a positive eigenvalue.
\end{proof}

The sign structure of $Q_{vv}(0)$ determines which forms of collapse are obstructed. Since $(L_{\mathcal{F}})_{vv}$ is positive semidefinite and $x_v x_v^T$ is rank one with eigenvalue $\|x_v\|^2$ in the direction of $x_v$, the matrix $Q_{vv}(0)$ is indefinite precisely when the Laplacian term dominates in directions orthogonal to $x_v$ while the opinion term dominates along $x_v$. In terms of the social dynamics, this means the agent's strongest convictions concern topics that are largely absent from their communication with neighbors. Colleagues who engage professionally while avoiding contested political views, or families that discuss everything except the one issue that actually matters, exhibit precisely this structure. The unexpressed convictions escape the pressure toward consensus, while the active communication channels remain engaged. The conservation law captures this stability: the system cannot resolve discord by either silencing communication or eliminating opinions. Equilibrium, therefore, represents genuine accommodation rather than disengagement.

\begin{remark}[Additional persistence mechanisms]
Beyond the conservation law, opinions may persist through other mechanisms. If $v \in U$ with $u_v \neq 0$, then $x_v(t) = u_v + y_v(t)$ remains bounded away from zero regardless of the dynamics. Similarly, when the restriction maps coupling stubborn to free vertices are frozen and the forcing term $L_{QS} u$ is nonzero, the free opinions cannot all vanish at equilibrium.
\end{remark}


\section{Timescale Separation and Stagnation Bounds}\label{sec:timescale}

The joint dynamics of Section~\ref{sec:joint} treat opinion and structure updates as occurring at comparable rates. In practice, these rates often differ substantially: private beliefs may shift slowly while public rhetoric adapts rapidly, or vice versa. This section analyzes what happens when the opinion update rate $\alpha$ and the structural adaptation rate $\beta$ are far apart.

The analysis mirrors classical singular perturbation theory. When $\beta \gg \alpha$, expressions act as the fast subsystem and equilibrate before opinions have time to change significantly. When $\alpha \gg \beta$, the roles reverse: opinions adapt rapidly while the communication structure remains nearly frozen. In each regime, we derive bounds on how much the slow variable can drift while the fast variable relaxes. Both bounds follow from a common structure: the fast dynamics drive exponential decay of the discrepancy energy, and integration of the slow velocity over this decay period yields the displacement bound.

\subsection{Fast Opinions}\label{ssec:fast-opinions}

In the regime $\alpha \gg \beta$, agents adjust their private opinions much faster than they alter their communication strategies. This describes a society of flexible individuals and rigid institutions: people readily compromise to fit the existing social structure, thereby preserving the network's communication patterns.

To quantify how much the structure can drift while opinions equilibrate, we introduce a measure of how efficiently opinion dynamics reduce disagreement. For a trajectory $(x(t), \delta(t))$ on a time interval $[0, T]$, define the \emph{effective spectral gap} (a trajectory-dependent dissipation ratio) of projected opinion diffusion:
\begin{equation}\label{eq:lambda_eff_def}
\lambda_{\mathrm{eff}} := \inf_{t \in [0,T]} \frac{\|P_{\mathcal{Q}}(L_{\mathcal{F}}(\delta(t)) x(t))\|^2}{\|\delta(t) x(t)\|^2}.
\end{equation}
This ratio compares the squared norm of the projected Laplacian (the opinion velocity, up to the rate $\alpha$) to the disagreement energy. We assume this infimum is strictly positive, i.e., the ratio is bounded away from zero whenever $\delta(t)x(t) \neq 0$. If disagreement vanishes at some time, the system is at equilibrium and the bounds below hold with zero displacement.

When $U = \varnothing$ (no stubborn vertices), $P_{\mathcal{Q}} = I$ and $\lambda_{\mathrm{eff}}$ is bounded below by the smallest positive eigenvalue of $L_{\mathcal{F}}(\delta)$, the sheaf's algebraic connectivity. When stubborn vertices are present, $\lambda_{\mathrm{eff}}$ depends on the interaction between the stubborn constraints and the sheaf structure.

\begin{proposition}[Structural Stagnation Bound]\label{prop:structural_stagnation}
Consider the constrained joint dynamics~\eqref{eq:constrained_joint} with initial condition $(y_0, \delta_0)$. Assume that on $[0, T]$:
\begin{enumerate}
\item opinions are bounded: $\sup_{t \in [0,T]}\|x(t)\| \le B_x$ for some $B_x > 0$,
\item the effective spectral gap~\eqref{eq:lambda_eff_def} is positive: $\lambda_{\mathrm{eff}} > 0$.
\end{enumerate}
Then the structural displacement satisfies
\begin{equation}\label{eq:structural_bound}
\|\delta(T) - \delta_0\|_F \le \frac{\beta B_x \|\delta_0 x_0\|}{\alpha \lambda_{\mathrm{eff}}} \bigl(1 - e^{-\alpha \lambda_{\mathrm{eff}} T}\bigr) \le \frac{\beta B_x \|\delta_0 x_0\|}{\alpha \lambda_{\mathrm{eff}}}.
\end{equation}
\end{proposition}

The bound has a natural interpretation. The ratio $\beta/\alpha$ controls how much structural drift is allowed. The effective spectral gap $\lambda_{\mathrm{eff}}$ determines the rate of this decay: larger gaps mean faster equilibration and less time for the structure to drift. The initial disagreement $\|\delta_0 x_0\|$ sets the scale of the problem.

\begin{proof}
The strategy is to establish exponential energy decay from the spectral gap, then integrate the structural velocity over the decay period.

The energy derivative satisfies $\frac{d\Psi}{dt} = \langle \delta x, \delta \frac{dx}{dt} \rangle + \langle \delta x, \frac{d\delta}{dt} x \rangle$. For the opinion contribution, using $\frac{dx}{dt} = -\alpha \iota_{\mathcal{Q}} P_{\mathcal{Q}}(L_{\mathcal{F}} x)$:
\[
\langle \delta x, \delta \frac{dx}{dt} \rangle = -\alpha \langle L_{\mathcal{F}} x, \iota_{\mathcal{Q}} P_{\mathcal{Q}}(L_{\mathcal{F}} x) \rangle = -\alpha \|P_{\mathcal{Q}}(L_{\mathcal{F}} x)\|^2,
\]
where the second equality uses $\langle z, \iota P z \rangle = \|Pz\|^2$ for any orthogonal projection $P$ with inclusion $\iota$. The structure contribution satisfies $\langle \delta x, \frac{d\delta}{dt} x \rangle \le 0$ by the gradient flow structure. Combining these with the definition of $\lambda_{\mathrm{eff}}$:
\[
\frac{d\Psi}{dt} \le -\alpha \|P_{\mathcal{Q}}(L_{\mathcal{F}} x)\|^2 \le -\alpha \lambda_{\mathrm{eff}} \|\delta x\|^2 = -2\alpha \lambda_{\mathrm{eff}} \Psi.
\]
Gr\"onwall's inequality yields $\Psi(t) \le \Psi(0) e^{-2\alpha \lambda_{\mathrm{eff}} t}$, hence $\|\delta(t) x(t)\| = \sqrt{2\Psi(t)} \le \|\delta_0 x_0\| e^{-\alpha \lambda_{\mathrm{eff}} t}$.

For the structural displacement, note that orthogonal projection does not increase the Frobenius norm, and for an outer product $\|uv^T\|_F = \|u\|\|v\|$. Thus the velocity satisfies 
\[
\|\frac{d\delta}{dt}\|_F = \beta\|\Pi_{\mathcal{M}}(\delta x\, x^T)\|_F \le \beta\|\delta x\| \|x\| \le \beta B_x \|\delta x\|.
\]
Integrating:
\begin{align*}
\|\delta(T) - \delta_0\|_F &\le \beta \int_0^T B_x \|\delta(t) x(t)\| \, dt \le \beta B_x \|\delta_0 x_0\| \int_0^T e^{-\alpha \lambda_{\mathrm{eff}} t} \, dt \\
&= \frac{\beta B_x \|\delta_0 x_0\|}{\alpha \lambda_{\mathrm{eff}}} \bigl(1 - e^{-\alpha \lambda_{\mathrm{eff}} T}\bigr). \qedhere
\end{align*}
\end{proof}

\begin{remark}[Stopping-time interpretation]\label{rmk:stopping_time}
Under the hypotheses of Proposition~\ref{prop:structural_stagnation}, suppose that $T$ is the first time such that $\|P_{\mathcal{Q}}(L_{\mathcal{F}}(\delta(t)) x(t))\| = \Delta$ for some threshold $\Delta>0$, and assume further that the infimum in~\eqref{eq:lambda_eff_def} is attained at $T$ (this is, for instance, the case if the map $t\mapsto \|P_{\mathcal{Q}}(L_{\mathcal{F}}(\delta(t))x(t))\|^2 / \|\delta(t)x(t)\|^2$ is monotone decreasing on $[0,T]$). Then the instantaneous ratio at time $T$,
\[
\frac{\|P_{\mathcal{Q}}(L_{\mathcal{F}}(\delta(T)) x(T))\|^2}{\|\delta(T) x(T)\|^2},
\]
equals $\lambda_{\mathrm{eff}}$, and hence the exponential factor can be rewritten in terms of $\Delta$ to yield
\[
\|\delta(T) - \delta_0\|_F \le \frac{\beta B_x}{\alpha \lambda_{\mathrm{eff}}} \left( \|\delta_0 x_0\| - \frac{\Delta}{\sqrt{\lambda_{\mathrm{eff}}}} \right).
\]
This sharper form therefore requires the additional attainment hypothesis.
\end{remark}

\subsection{Fast Structure}\label{ssec:fast-structure}

When $\beta \gg \alpha$, communication structures adapt rapidly to minimize discord while private opinions barely change. Agents mask disagreement through diplomatic expression rather than resolving it through genuine opinion change. The result is a kind of surface harmony: public discourse appears harmonious, but private convictions remain largely unchanged.

As before, we introduce a measure of how efficiently structure dynamics reduce disagreement. For a trajectory $(x(t), \delta(t))$ on $[0, T]$, define the \emph{effective spectral gap} (again a trajectory-dependent dissipation ratio) of projected structure diffusion:
\begin{equation}\label{eq:mu_eff_def}
\mu_{\mathrm{eff}} := \inf_{t \in [0,T]} \frac{\|\Pi_{\mathcal{M}}(\delta(t) x(t) x(t)^T)\|_F^2}{\|\delta(t) x(t)\|^2} = \inf_{t \in [0,T]} \frac{\sum_{(v,e) \in \mathcal{I}} \|(\delta x)_e\|^2 \|x_v\|^2}{\|\delta x\|^2}.
\end{equation}
This ratio compares the squared Frobenius norm of the projected structure velocity to the disagreement energy. The second equality shows that $\mu_{\mathrm{eff}}$ is a disagreement-weighted average of squared opinion norms at adapting vertices. We assume this infimum is strictly positive, i.e., the ratio is bounded away from zero whenever $\delta(t)x(t) \neq 0$. If disagreement vanishes at some time, the system is at equilibrium and the bounds below hold with zero displacement.

When all incidences adapt and opinions are approximately uniform ($\|x_v\| \approx c$ for all $v$), then $\mu_{\mathrm{eff}} \approx 2c^2$ since each edge has two incidences. Conversely, $\mu_{\mathrm{eff}}$ approaches zero when disagreement occurs primarily on edges where adapting endpoints have weak opinions, or on edges with no adapting incidences at all.

\begin{proposition}[Opinion Stagnation Bound]\label{prop:opinion_stagnation}
Consider the constrained joint dynamics~\eqref{eq:constrained_joint} with initial condition $(y_0, \delta_0)$. Assume that on $[0, T]$:
\begin{enumerate}
\item restriction maps are bounded: $\sup_{t \in [0,T]}\|\delta(t)\|_F \le B_\delta$ for some $B_\delta > 0$,
\item the effective spectral gap~\eqref{eq:mu_eff_def} is positive: $\mu_{\mathrm{eff}} > 0$.
\end{enumerate}
Then the opinion displacement satisfies
\begin{equation}\label{eq:opinion_bound}
\|x(T) - x(0)\| \le \frac{\alpha B_\delta \|\delta_0 x_0\|}{\beta \mu_{\mathrm{eff}}} \bigl(1 - e^{-\beta \mu_{\mathrm{eff}} T}\bigr) \le \frac{\alpha B_\delta \|\delta_0 x_0\|}{\beta \mu_{\mathrm{eff}}}.
\end{equation}
\end{proposition}

The bound quantifies how fast is too fast: condition~\eqref{eq:opinion_bound} shows that the rate ratio $\alpha/\beta$ controls how much genuine opinion change occurs while the structure equilibrates.

Hypothesis (1) holds automatically with $B_\delta = \|\delta_0\|_F$ when $E_A = \varnothing$ (no Type A edges), since the Frobenius norm is non-increasing under such dynamics by Theorem~\ref{thm:constrained_convergence}. When Type A edges are present, boundedness must be verified separately or ensured via regularization (Appendix~\ref{app:regularization}).

\begin{proof}
The strategy mirrors the previous proposition: establish exponential energy decay, then integrate the opinion velocity.

For the structure contribution to the energy derivative, the dynamics give $\frac{d}{dt}(\delta_{v \unlhd e}) = -\beta (\delta x)_e x_v^T$ for each adapting incidence $(v, e) \in \mathcal{I}$. A direct computation shows:
\[
\langle \delta x, \frac{d\delta}{dt} x \rangle = -\beta \sum_{(v,e) \in \mathcal{I}} \|(\delta x)_e\|^2 \|x_v\|^2 = -\beta \|\Pi_{\mathcal{M}}(\delta x\, x^T)\|_F^2.
\]
The opinion contribution satisfies $\langle \delta x, \delta \frac{dx}{dt} \rangle \le 0$. Combining with the definition of $\mu_{\mathrm{eff}}$:
\[
\frac{d\Psi}{dt} \le -\beta \|\Pi_{\mathcal{M}}(\delta x\, x^T)\|_F^2 \le -\beta \mu_{\mathrm{eff}} \|\delta x\|^2 = -2\beta \mu_{\mathrm{eff}} \Psi.
\]
Gr\"onwall's inequality yields $\Psi(t) \le \Psi(0) e^{-2\beta \mu_{\mathrm{eff}} t}$, hence $\|\delta(t) x(t)\| \le \|\delta_0 x_0\| e^{-\beta \mu_{\mathrm{eff}} t}$.

For the opinion displacement, note that $\|P_{\mathcal{Q}}(L_{\mathcal{F}} x)\| \le \|L_{\mathcal{F}} x\|$ since orthogonal projection does not increase the norm, and $L_{\mathcal{F}} = \delta^T \delta$ gives $\|L_{\mathcal{F}} x\| \le \|\delta^T\|_{\mathrm{op}} \|\delta x\| \le \|\delta\|_F \|\delta x\|$. Thus, the velocity satisfies:
\[
\|\frac{dy}{dt}\| = \alpha \|P_{\mathcal{Q}}(L_{\mathcal{F}} x)\| \le \alpha \|\delta\|_F \|\delta x\| \le \alpha B_\delta \|\delta x\|.
\]
Since $\|x(T) - x(0)\| = \|y(T) - y_0\|$:
\begin{align*}
\|x(T) - x(0)\| &\le \alpha B_\delta \int_0^T \|\delta(t) x(t)\| \, dt \le \frac{\alpha B_\delta \|\delta_0 x_0\|}{\beta \mu_{\mathrm{eff}}} \bigl(1 - e^{-\beta \mu_{\mathrm{eff}} T}\bigr). \qedhere
\end{align*}
\end{proof}

As with the structural stagnation bound, a sharper $\Delta$-dependent form can be obtained under additional monotonicity and attainment hypotheses; see Remark~\ref{rmk:stopping_time}.

When $\mu_{\mathrm{eff}} = 0$, the bound becomes ineffective. This occurs when disagreement concentrates on edges where adapting endpoints have vanishing opinions, or on edges with no adapting incidences. In such cases, regularization provides an unconditional bound; see Appendix~\ref{app:regularization}.

\begin{remark}[Unified perspective]
Propositions~\ref{prop:structural_stagnation} and~\ref{prop:opinion_stagnation} share a common structure. In both cases, projected diffusion on the fast variable drives exponential energy decay, and integration of the slow velocity yields the displacement bound. The effective spectral gaps $\lambda_{\mathrm{eff}}$ and $\mu_{\mathrm{eff}}$ play symmetric roles: the former measures the efficiency of projected diffusion on the discourse sheaf $\mathcal{F}$, while the latter measures the efficiency of projected diffusion on the sheaf of structures $\mathcal{H}^x$. Both bounds take the form
\[
\text{(slow displacement)} \le \frac{\text{(slow rate)} \times \text{(slow variable bound)} \times \|\delta_0 x_0\|}{\text{(fast rate)} \times \text{(effective gap)}}.
\]
\end{remark}

\begin{corollary}[Regime characterization]\label{cor:regime_thresholds}
Fix a tolerance $\varepsilon > 0$ and define the threshold ratios
\begin{equation}\label{eq:threshold_ratios}
\rho_- := \frac{\varepsilon \lambda_{\mathrm{eff}}}{B_x \|\delta_0 x_0\|}, \qquad \rho_+ := \frac{B_\delta \|\delta_0 x_0\|}{\varepsilon \mu_{\mathrm{eff}}}.
\end{equation}
Under the hypotheses of Propositions~\ref{prop:structural_stagnation} and~\ref{prop:opinion_stagnation}:
if $\beta/\alpha \le \rho_-$, then structural displacement is bounded by $\varepsilon$ (belief resolution);
if $\beta/\alpha \ge \rho_+$, then opinion displacement is bounded by $\varepsilon$ (rhetorical accommodation).
When $\varepsilon^2 < B_x B_\delta \|\delta_0 x_0\|^2 / (\lambda_{\mathrm{eff}} \mu_{\mathrm{eff}})$, we have $\rho_- < \rho_+$. In the intermediate regime $\rho_- < \beta/\alpha < \rho_+$, the upper bounds on structural and opinion displacement both exceed $\varepsilon$, so this regime allows (but does not force) both variables to change by more than $\varepsilon$.
\end{corollary}


\section{Conclusion}\label{sec:conclusion}

We have developed a theory of stubborn agents on discourse sheaves that distinguishes between rigidity in private beliefs and rigidity in public expression. The central insight is that these two forms of stubbornness are governed by parallel cohomological structures. Stubborn beliefs reduce opinion dynamics to a sheaf Poisson equation on the sheaf of free opinions, while stubborn expressions reduce structure dynamics to Laplacian flow on an auxiliary sheaf of free structures. In both cases, constrained dynamics take the form of gradient descent with affine forcing, and equilibria admit equivalent characterizations as projections, corrections, and variational solutions.

When beliefs and expressions evolve simultaneously, four scenarios emerge from the possible configurations of frozen and adapting communication channels. Under suitable assumptions, trajectories converge, but perfect agreement is achievable only when the effective Laplacian has trivial kernel. Conservation laws ensure that equilibrium agreement, when achieved, reflects genuine accommodation rather than disengagement: the system cannot reach spurious consensus through vanishing opinions or ceasing communication. When update rates differ substantially, timescale separation determines which variable stagnates. If expressions adapt faster than beliefs, agents reach surface harmony before private convictions have had time to evolve. The reverse regime produces the opposite effect: flexible individuals conform to rigid communication norms while the underlying discourse structure remains unchanged.

Several directions remain open. The stagnation bounds derived here are sufficient conditions; sharper characterizations of the boundary between regimes would be valuable. On the modeling side, the present framework assumes fixed stalks and pairwise interactions. Allowing stalks to evolve would capture how agents' opinion spaces change as interests shift or new dimensions of discourse become salient~\cite{ahn2020crosscoupling}. Quiver representations \cite{sumray2024quiver} offer a natural generalization to group interactions and self-discourse, where an individual reconciles conflicting internal positions. The framework developed here treats stubbornness as a constraint on dynamics rather than a property of agents. This perspective may prove useful beyond opinion dynamics, wherever selective rigidity shapes the evolution of networked systems.





\bibliographystyle{abbrvnat}
\bibliography{references}  


\appendix

\section{Exact Sequences}\label{app:exact-sequences}

The constrained opinion dynamics of Section~\ref{sec:stubborn-opinions} arise from a subsheaf inclusion. This appendix develops the exact sequence structure underlying the sheaf of free opinions $\mathcal{Q}$.

The sheaf $\mathcal{Q}$ of free opinions is a subsheaf of the discourse sheaf $\mathcal{F}$. The inclusion $\phi: \mathcal{Q} \hookrightarrow \mathcal{F}$ is defined by $\phi_v = \iota_v: T_v \hookrightarrow \mathcal{F}(v)$ at vertices and $\phi_e = \mathrm{id}_{\mathcal{F}(e)}$ at edges. For each incidence $v \unlhd e$, the diagram
\[
\begin{tikzcd}
T_v \arrow[r, hook, "\iota_v"] \arrow[d, "\mathcal{Q}_{v \unlhd e}"'] & \mathcal{F}(v) \arrow[d, "\mathcal{F}_{v \unlhd e}"] \\
\mathcal{F}(e) \arrow[r, equal] & \mathcal{F}(e)
\end{tikzcd}
\]
commutes by definition of $\mathcal{Q}_{v \unlhd e} = \mathcal{F}_{v \unlhd e} \circ \iota_v$.

The quotient sheaf $\mathcal{S} = \mathcal{F}/\mathcal{Q}$ has stalks $\mathcal{S}(v) = \mathcal{F}(v)/T_v = S_v$ at vertices. At edges, $\mathcal{S}(e) = \mathcal{F}(e)/\mathcal{Q}(e) = 0$ since $\mathcal{Q}(e) = \mathcal{F}(e)$. The vanishing of edge stalks forces all restriction maps of $\mathcal{S}$ to be zero, hence $\delta^{\mathcal{S}} = 0$.

\begin{proposition}[Long exact sequence for stubborn opinions]
\label{prop:long-exact-seq}
The short exact sequence of sheaves $0 \to \mathcal{Q} \xrightarrow{\phi} \mathcal{F} \xrightarrow{\pi} \mathcal{S} \to 0$ induces a long exact sequence in cohomology:
\[
0 \to H^0(G; \mathcal{Q}) \xrightarrow{\phi_*} H^0(G; \mathcal{F}) \xrightarrow{\pi_*} C^0(S) \xrightarrow{\partial} H^1(G; \mathcal{Q}) \to H^1(G; \mathcal{F}) \to 0,
\]
where we have used $H^0(G; \mathcal{S}) = C^0(S)$ and $H^1(G; \mathcal{S}) = 0$.
\end{proposition}

\begin{proof}
The identifications $H^0(G; \mathcal{S}) = C^0(S)$ and $H^1(G; \mathcal{S}) = 0$ follow from $\delta^{\mathcal{S}} = 0$: we have $H^0(G; \mathcal{S}) = \ker(\delta^{\mathcal{S}}) = C^0(\mathcal{S}) = C^0(S)$, and $H^1(G; \mathcal{S}) = C^1(\mathcal{S})/\mathrm{im}(\delta^{\mathcal{S}}) = 0$ since $C^1(\mathcal{S}) = 0$. The long exact sequence then follows from standard sheaf cohomology.
\end{proof}

The connecting homomorphism $\partial: C^0(S) \to H^1(G; \mathcal{Q})$ admits an explicit description. For a stubborn configuration $u \in C^0(S)$, lift to $\tilde{u} = \iota_S(u) \in C^0(G; \mathcal{F})$ and apply the coboundary to obtain $\delta^{\mathcal{F}} \tilde{u} \in C^1(G; \mathcal{F}) = C^1(\mathcal{Q})$. The class $\partial(u) = [\delta^{\mathcal{F}} \tilde{u}] \in H^1(G; \mathcal{Q})$ measures the obstruction to extending $u$ to a global section: by exactness at $C^0(S)$, we have $\partial(u) = 0$ if and only if there exists $x^* \in H^0(G; \mathcal{F})$ with $P_S(x^*) = u$.

This criterion has a natural interpretation in terms of opinion dynamics. When $\partial(u) = 0$, the stubborn configuration is compatible: there exists an opinion distribution achieving perfect public agreement among all agents, stubborn and free alike. The stubborn agents happen to hold positions that can be harmonized. When $\partial(u) \neq 0$, the stubborn configuration is incompatible: no matter what opinions free agents adopt, residual disagreement is unavoidable. The Poisson equation still admits a solution, but free agents minimize disagreement rather than eliminate it. The cohomology class $\partial(u) \in H^1(G; \mathcal{Q})$ thus measures the obstruction to harmony imposed by stubborn beliefs.

The injectivity of $\phi_*: H^0(G; \mathcal{Q}) \to H^0(G; \mathcal{F})$ has a direct consequence for the Poisson equation of Theorem~\ref{thm:stubborn-directions}: any global section of $\mathcal{Q}$ extends to a global section of $\mathcal{F}$. This implies $\ker(L_{\mathcal{Q}}) \subseteq \ker(L_{SQ})$, ensuring solvability.

The sheaf of free structures $\mathcal{H}^x_{\mathcal{I}}$ from Section~\ref{ssec:sheaf-free-structures} is likewise a subsheaf of $\mathcal{H}^x$, so a similar exact sequence structure applies. The following section examines the categorical construction of $\mathcal{H}^x$ itself and the bilinear duality between opinions and structures.


\section{Categorical Construction of the Sheaf of Structures}\label{app:categorical-structure}

The sheaf of structures $\mathcal{H}^x$ introduced in Definition~\ref{def:sheaf-of-structures} exhibits a natural duality with the opinion sheaf $\mathcal{F}$. This appendix examines the categorical underpinnings of this construction: its generality, the precise sense in which $\mathcal{F}$ and $\mathcal{H}^x$ are dual, and a canonical embedding that clarifies their relationship.

\subsection{Generality of the Construction}\label{app:generality}

The construction of $\mathcal{H}^x$ requires three categorical ingredients. First, for each incidence $v \unlhd e$, the space $\mathrm{Hom}(\mathcal{F}(v), \mathcal{F}(e))$ must exist as an object in the same category as the stalks. Second, at vertices with multiple incident edges, the vertex stalk 
\[
\mathcal{H}^x(v) = \bigoplus_{e : v \unlhd e} \mathrm{Hom}(\mathcal{F}(v), \mathcal{F}(e))
\] 
requires finite direct sums. Third, the restriction maps must be valid morphisms in the category. Concretely, the restriction to edge $e$ factors as $(\mathcal{H}^x)_{v \unlhd e} = \mathrm{ev}_{x_v} \circ \pi_e$, where $\pi_e: \mathcal{H}^x(v) \to \mathrm{Hom}(\mathcal{F}(v), \mathcal{F}(e))$ projects onto the $e$-component and $\mathrm{ev}_{x_v}(\rho) = \rho(x_v)$ evaluates at the fixed opinion.

Any \emph{additive category with internal Hom} satisfies these requirements. In particular, the construction applies to vector spaces over any field $k$ (the setting of this paper), modules over a commutative ring $R$, and abelian groups. It does \emph{not} extend naturally to sheaves valued in sets, where Hom-objects lack the necessary structure, or to topological spaces without linear structure.

The categorical construction of $\mathcal{H}^x$ is purely algebraic and does not require inner products. The analytical results of this paper (least-squares characterizations, gradient flows, and convergence theorems) do require inner products on the stalks, as assumed throughout the main text. Under this assumption, $\mathrm{Hom}(\mathcal{F}(v), \mathcal{F}(e))$ inherits the Frobenius inner product.

\subsection{The Duality Between $\mathcal{F}$ and $\mathcal{H}^x$}\label{app:duality}

A natural question is whether the relationship between the opinion sheaf $\mathcal{F}$ and the structure sheaf $\mathcal{H}^x$ can be expressed as a sheaf morphism. Recall that a sheaf morphism $\varphi: \mathcal{F} \to \mathcal{G}$ consists of linear maps $\varphi_\sigma: \mathcal{F}(\sigma) \to \mathcal{G}(\sigma)$ for each cell $\sigma$, such that for every incidence $v \unlhd e$, the diagram
\[
\begin{tikzcd}
\mathcal{F}(v) \arrow[r, "\varphi_v"] \arrow[d, "\mathcal{F}_{v \unlhd e}"'] & \mathcal{G}(v) \arrow[d, "\mathcal{G}_{v \unlhd e}"] \\
\mathcal{F}(e) \arrow[r, "\varphi_e"'] & \mathcal{G}(e)
\end{tikzcd}
\]
commutes. A morphism $\varphi: \mathcal{F} \to \mathcal{H}^x$ would require:
\[
\begin{tikzcd}
\mathcal{F}(v) \arrow[r, dashed] \arrow[d, "\mathcal{F}_{v \unlhd e}"'] & \mathrm{Hom}(\mathcal{F}(v), \mathcal{F}(e)) \arrow[d, "\mathrm{ev}_{x_v}"] \\
\mathcal{F}(e) \arrow[r, dashed] & \mathcal{F}(e)
\end{tikzcd}
\]
The restriction map $\mathrm{ev}_{x_v}$ depends on the opinion $x_v$, so $\mathcal{H}^x$ is not a fixed target sheaf. The relationship between $\mathcal{F}$ and $\mathcal{H}^x$ is therefore not a sheaf morphism but a bilinear pairing.

For an edge $e = u \sim v$, an opinion cochain $x \in C^0(G; \mathcal{F})$, and a configuration of restriction maps $\rho$, define the \emph{discrepancy pairing}
\[
\langle \rho, x \rangle_e := \rho_{v \unlhd e}(x_v) - \rho_{u \unlhd e}(x_u).
\]
This pairing captures the edge discrepancy, and satisfies the central identity
\begin{equation}\label{eq:duality-identity}
(\delta^{\mathcal{F}^\rho} x)_e = (\delta^{\mathcal{H}^x} \rho)_e = \langle \rho, x \rangle_e,
\end{equation}
where $\mathcal{F}^\rho$ denotes the sheaf with restriction maps $\rho$. The coboundary of $x$ in the sheaf determined by $\rho$ equals the coboundary of $\rho$ in the sheaf determined by $x$. This identity is the correct expression of duality between opinions and structures.

\subsection{The Bidual Embedding}\label{app:bidual}

Although there is no canonical sheaf morphism from $\mathcal{F}$ to $\mathcal{H}^x$, there is a canonical embedding of $\mathcal{F}$ into a related sheaf that clarifies the structure. This construction is analogous to the bidual embedding $V \hookrightarrow V^{**}$ from linear algebra, where $v \mapsto \mathrm{ev}_v$ embeds a vector space into its double dual.

\begin{definition}[Bidual sheaf]\label{def:bidual-sheaf}
Let $\mathcal{F}$ be a discourse sheaf on $G$. For each vertex $v$, define
\[
\mathcal{H}(v) := \bigoplus_{e : v \unlhd e} \mathrm{Hom}(\mathcal{F}(v), \mathcal{F}(e)), \qquad E_v := \bigoplus_{e : v \unlhd e} \mathcal{F}(e).
\]
The \emph{bidual sheaf} $\mathcal{K}$ is the cellular sheaf on $G$ with:
\begin{itemize}
\item Vertex stalks: $\mathcal{K}(v) := \mathrm{Hom}(\mathcal{H}(v), E_v)$.
\item Edge stalks: $\mathcal{K}(e) := \mathcal{F}(e)$.
\item Restriction maps: For each incidence $v \unlhd e$ and each $\Psi \in \mathcal{K}(v)$,
\[
\mathcal{K}_{v \unlhd e}(\Psi) := \pi_e\bigl(\Psi(\mathcal{F}_v)\bigr),
\]
where $\mathcal{F}_v := (\mathcal{F}_{v \unlhd e'})_{e' : v \unlhd e'} \in \mathcal{H}(v)$ is the tuple of restriction maps of $\mathcal{F}$ at $v$, and $\pi_e: E_v \to \mathcal{F}(e)$ projects onto the $e$-component.
\end{itemize}
\end{definition}

The tuple $\mathcal{F}_v$ encodes the fixed restriction maps of the original sheaf at vertex $v$. Incorporating this fixed data into $\mathcal{K}$'s restriction maps is the key to the construction: the target sheaf $\mathcal{K}$ does not depend on any opinion $x$, so a morphism $\mathcal{F} \to \mathcal{K}$ can be defined.

\begin{proposition}[Canonical embedding]\label{prop:canonical-embedding}
The maps
\[
\varphi_v: \mathcal{F}(v) \to \mathcal{K}(v), \quad x_v \mapsto \mathrm{ev}_{x_v},
\]
where $\mathrm{ev}_{x_v}: \mathcal{H}(v) \to E_v$ evaluates each component at $x_v$, together with $\varphi_e := \mathrm{id}_{\mathcal{F}(e)}$, define a sheaf morphism $\varphi: \mathcal{F} \to \mathcal{K}$. That is, for every incidence $v \unlhd e$, the diagram
\[
\begin{tikzcd}
\mathcal{F}(v) \arrow[r, "\varphi_v"] \arrow[d, "\mathcal{F}_{v \unlhd e}"'] & \mathcal{K}(v) \arrow[d, "\mathcal{K}_{v \unlhd e}"] \\
\mathcal{F}(e) \arrow[r, "\mathrm{id}"'] & \mathcal{K}(e)
\end{tikzcd}
\]
commutes.
\end{proposition}

\begin{proof}
We verify linearity and commutativity.

\emph{Linearity.} For $x, y \in \mathcal{F}(v)$ and $\lambda \in \mathbb{R}$, evaluation is linear in the point: $\mathrm{ev}_{x+y} = \mathrm{ev}_x + \mathrm{ev}_y$ and $\mathrm{ev}_{\lambda x} = \lambda \, \mathrm{ev}_x$.

\emph{Commutativity.} Fix an incidence $v \unlhd e$. The left-down path gives
\[
\mathcal{K}_{v \unlhd e}(\varphi_v(x_v)) = \pi_e(\mathrm{ev}_{x_v}(\mathcal{F}_v)) = \pi_e((\mathcal{F}_{v \unlhd e'}(x_v))_{e'}) = \mathcal{F}_{v \unlhd e}(x_v).
\]
The right-down path gives $\varphi_{e}(\mathcal{F}_{v \unlhd e}(x_v)) = \mathcal{F}_{v \unlhd e}(x_v)$, so the two paths agree.
\end{proof}

The embedding $\varphi$ is injective whenever the maps in $\mathcal{H}(v)$ separate points of $\mathcal{F}(v)$, meaning distinct opinions yield distinct evaluation maps. In finite dimensions, this holds as long as some incident edge has a nonzero stalk. Under this mild condition, $\varphi$ embeds $\mathcal{F}$ as a subsheaf of $\mathcal{K}$.

The bidual sheaf $\mathcal{K}$ does not coincide with the structure sheaf $\mathcal{H}^x$, but they share a common ingredient: $\mathcal{H}(v)$ appears as the vertex stalk of $\mathcal{H}^x$ and as the domain in $\mathcal{K}(v) = \mathrm{Hom}(\mathcal{H}(v), E_v)$. The morphism $\varphi: \mathcal{F} \to \mathcal{K}$ embeds opinions into a fixed target sheaf, complementing the bilinear pairing of Section~\ref{app:duality}, which captures the interaction between opinions and structures directly.


\section{Regularization for Joint Dynamics}\label{app:regularization}

The constrained joint dynamics of Section~\ref{sec:joint} converge under favorable conditions, but two pathologies can prevent global boundedness. First, when Type A edges are present ($E_A \neq \varnothing$), the pursuit of zero disagreement can cause evolving restriction maps to increase in magnitude: the Frobenius norm $\|\delta\|_F$ may grow even as the disagreement energy $\Psi$ decreases (see Example~\ref{ex:typeA-growth}).

\begin{example}[Norm growth under Type A adaptation]\label{ex:typeA-growth}
Consider a single edge $e = u \sim v$ with scalar stalks $\mathcal{F}(u) = \mathcal{F}(v) = \mathcal{F}(e) = \mathbb{R}$. Let $x_u = 1$ be stubborn, $x_v$ free, and suppose the incidence $(u, e)$ adapts while $(v, e)$ is frozen with $\mathcal{F}_{v \unlhd e} = 1$. Writing $p = \mathcal{F}_{u \unlhd e}$, the discrepancy is $d = x_v - p$ and the dynamics reduce to
\[
\frac{dx_v}{dt} = -\alpha d, \qquad \frac{dp}{dt} = \beta d.
\]
The discrepancy satisfies $\frac{d}{dt}(d) = -(\alpha + \beta)d$, so $d(t) = d_0 e^{-(\alpha+\beta)t}$ decays exponentially. Meanwhile, $p(t) = p_0 + \frac{\beta d_0}{\alpha + \beta}(1 - e^{-(\alpha+\beta)t})$. With initial conditions $p_0 = 1$, $x_v(0) = 2$, and $\alpha = \beta = 1$:
\[
\|\delta(0)\|_F^2 = p_0^2 + 1 = 2, \qquad \|\delta(\infty)\|_F^2 = \Bigl(1 + \tfrac{1}{2}\Bigr)^2 + 1 = \tfrac{13}{4}.
\]
The Frobenius norm increases by approximately $28\%$ despite the disagreement energy $\Psi = \frac{1}{2}d^2$ decreasing to zero. In networks with multiple Type A edges, such growth can accumulate, motivating the regularization introduced below.
\end{example}

\subsection{Convergence results}

To ensure global boundedness and guarantee convergence without additional structural assumptions, we introduce regularization terms that penalize deviation from the initial state. The \emph{regularized energy functional} is
\begin{align}\label{eq:regularized_energy_app}
\mathcal{L}_{\lambda,\mu}(y, \delta) = \frac{1}{2} \| \delta x \|^2 + \frac{\lambda}{2} \| y - y_0 \|^2 + \frac{\mu}{2} \| \delta - \delta_0 \|_F^2,
\end{align}
where $x = \tilde{u} + \tilde{y}$ is the total opinion state and $(y_0, \delta_0)$ denotes the initial condition. The first term measures disagreement, while the second and third terms penalize deviation of opinions and communication structure from their initial values, respectively. The corresponding gradient descent dynamics are
\begin{equation}\label{eq:regularized_dynamics_app}
\begin{aligned}
\frac{dy}{dt} &= -\alpha \Bigl( P_{\mathcal{Q}} \bigl( \delta^T \delta x \bigr) + \lambda (y - y_0) \Bigr), \\
\frac{d\delta}{dt} &= -\beta \, \Pi_{\mathcal{M}} \Bigl( \delta x\, x^T + \mu (\delta - \delta_0) \Bigr).
\end{aligned}
\end{equation}

\begin{proposition}[Convergence of Regularized Dynamics]\label{prop:regularization}
For any $\lambda, \mu > 0$, the trajectories of the regularized dynamics \eqref{eq:regularized_dynamics_app} satisfy:
\begin{enumerate}
    \item The regularized energy $\mathcal{L}_{\lambda,\mu}$ is nonincreasing along trajectories.
    \item The trajectories are globally bounded and precompact.
    \item The system converges to an equilibrium $(y^*, \delta^*)$ satisfying
    \begin{align}
    P_{\mathcal{Q}}(\delta^{*T} \delta^* x^*) + \lambda(y^* - y_0) &= 0, \\
    \Pi_{\mathcal{M}}\bigl(\delta^* x^* (x^*)^T + \mu(\delta^* - \delta_0)\bigr) &= 0.
    \end{align}
\end{enumerate}
\end{proposition}

\begin{proof}
We establish the three claims in turn.

For energy dissipation, the time derivative of $\mathcal{L}_{\lambda,\mu}$ along trajectories is
\[
\frac{d\mathcal{L}_{\lambda,\mu}}{dt} = \langle \delta x, \frac{d\delta}{dt} x + \delta \frac{dx}{dt} \rangle + \lambda \langle y - y_0, \frac{dy}{dt} \rangle + \mu \langle \delta - \delta_0, \frac{d\delta}{dt} \rangle_F.
\]
Using the cyclic property of the trace, the first term becomes $\langle \delta x, \frac{d\delta}{dt} x \rangle = \langle \frac{d\delta}{dt}, \delta x\, x^T \rangle_F$. For the second term, since $\frac{dx}{dt} = \iota_{\mathcal{Q}}(\frac{dy}{dt})$, we have
\[
\langle \delta x, \delta \iota_{\mathcal{Q}}(\frac{dy}{dt}) \rangle = \langle \delta^T \delta x, \iota_{\mathcal{Q}}(\frac{dy}{dt}) \rangle = \langle P_{\mathcal{Q}}(\delta^T \delta x), \frac{dy}{dt} \rangle,
\]
where the last equality uses $P_{\mathcal{Q}} = \iota_{\mathcal{Q}}^T$ when restricted to $\mathrm{im}(\iota_{\mathcal{Q}})$. Collecting terms gives
\[
\frac{d\mathcal{L}_{\lambda,\mu}}{dt} = \langle \frac{d\delta}{dt}, \delta x\, x^T + \mu(\delta - \delta_0) \rangle_F + \langle \frac{dy}{dt}, P_{\mathcal{Q}}(\delta^T \delta x) + \lambda(y - y_0) \rangle.
\]
Substituting the dynamics from \eqref{eq:regularized_dynamics_app} and using the self-adjoint property of orthogonal projections yields
\[
\frac{d\mathcal{L}_{\lambda,\mu}}{dt} = -\beta \| \Pi_{\mathcal{M}}(\delta x\, x^T + \mu(\delta - \delta_0)) \|_F^2 - \alpha \| P_{\mathcal{Q}}(\delta^T \delta x) + \lambda(y - y_0) \|^2 \leq 0.
\]

For boundedness, since $\mathcal{L}_{\lambda,\mu}(t) \leq \mathcal{L}_{\lambda,\mu}(0)$ for all $t \geq 0$, each term in \eqref{eq:regularized_energy_app} is individually bounded by $\mathcal{L}_{\lambda,\mu}(0)$. In particular,
\begin{align}\label{eq:regularized_bounds}
\|y(t) - y_0\|^2 \leq \frac{2\mathcal{L}_{\lambda,\mu}(0)}{\lambda}, \qquad \|\delta(t) - \delta_0\|_F^2 \leq \frac{2\mathcal{L}_{\lambda,\mu}(0)}{\mu}.
\end{align}
Thus both $y(t)$ and $\delta(t)$ remain in bounded sets, and the forward orbit is precompact.

For convergence, LaSalle's Invariance Principle applies to the precompact orbit with nonincreasing Lyapunov function $\mathcal{L}_{\lambda,\mu}$. The trajectory converges to the largest invariant subset of $\{d\mathcal{L}_{\lambda,\mu}/dt = 0\}$. Since both squared terms in the dissipation must vanish, this invariant set is precisely the equilibrium set characterized by the stated conditions.
\end{proof}

The regularization terms model agents' \emph{reluctance to change}. The parameter $\lambda$ captures cognitive inertia in updating beliefs, while $\mu$ captures the cost of adapting established communication patterns. At equilibrium, the condition $P_{\mathcal{Q}}(\delta^{*T} \delta^* x^*) = -\lambda(y^* - y_0)$ shows that agents tolerate some social pressure rather than fully adapting their opinions. Similarly, the structure equilibrium shows that agents tolerate some residual disagreement rather than fully adapting their communication. Small $\lambda, \mu$ favor harmony over authenticity; large values favor authenticity over harmony.

In the limit $\lambda, \mu \to 0$, the regularized dynamics formally reduce to the unregularized dynamics \eqref{eq:constrained_joint}. When the unregularized system has bounded trajectories, the equilibria coincide in this limit.

\subsection{Unconditional Stagnation Bounds}

The stagnation bounds of Section~\ref{sec:timescale} require either positive effective spectral gap $\lambda_{\mathrm{eff}} > 0$ (for structural stagnation) or positive effective spectral gap $\mu_{\mathrm{eff}} > 0$ (for opinion stagnation). When these conditions fail, regularization provides unconditional bounds.

\begin{corollary}[Structural Bound with Stubborn Vertices]\label{cor:structural_stagnation_stubborn}
Consider the regularized joint dynamics \eqref{eq:regularized_dynamics_app} with parameters $\lambda, \mu > 0$ and initial condition $(y_0, \delta_0)$. Then
\[
\sup_{t \ge 0} \|\delta(t) - \delta_0\|_F \le \sqrt{\frac{2\mathcal{L}_{\lambda,\mu}(0)}{\mu}}.
\]
\end{corollary}

\begin{proof}
Immediate from \eqref{eq:regularized_bounds}.
\end{proof}

This bound applies regardless of whether stubborn vertices are present. When $U \neq \varnothing$, the opinion dynamics cannot drive disagreement to zero on edges incident to stubborn vertices, so the structural velocity remains nonzero indefinitely. The regularization term $\frac{\mu}{2}\|\delta - \delta_0\|_F^2$ arrests this drift by penalizing deviation from the initial communication structure.

\begin{corollary}[Opinion Bound with Frozen Edges]\label{cor:opinion_stagnation_frozen}
Consider the regularized joint dynamics \eqref{eq:regularized_dynamics_app} with parameters $\lambda, \mu > 0$ and initial condition $(y_0, \delta_0)$. Then
\[
\sup_{t \ge 0} \|y(t) - y_0\| \le \sqrt{\frac{2\mathcal{L}_{\lambda,\mu}(0)}{\lambda}}.
\]
\end{corollary}

\begin{proof}
Immediate from \eqref{eq:regularized_bounds}.
\end{proof}

This bound applies regardless of how edges are frozen. When many edges incident to free vertices have frozen restriction maps, the effective spectral gap $\mu_{\mathrm{eff}}$ in Proposition~\ref{prop:opinion_stagnation} approaches zero and that bound becomes ineffective. The regularization term $\frac{\lambda}{2}\|y - y_0\|^2$ arrests opinion drift by penalizing deviation from initial beliefs.

Together, these corollaries show that regularization provides a universal safety net: for any choice of stubborn vertices and frozen edges, the regularized dynamics keep both opinions and structure within distance $O(1/\sqrt{\lambda})$ and $O(1/\sqrt{\mu})$ of their initial values.


\section{Equilibrium Calculations for the Four Scenarios}\label{app:example-calculations}

We compute the equilibria of the joint dynamics for the single-edge example from Figure~\ref{fig:four-scenarios}. The edge $e = \{u, v\}$ has stalks $\mathcal{F}(u) = \mathbb{R}^2$, $\mathcal{F}(v) = \mathcal{F}(e) = \mathbb{R}$, with the first coordinate of $\mathcal{F}(u)$ stubborn. Writing $\mathcal{F}_{u \unlhd e} = [a \; b]$ and $\mathcal{F}_{v \unlhd e} = [c]$, the initial configuration is
\[
x_u = \begin{bmatrix} 1 \\ 1 \end{bmatrix}, \quad x_v = -1, \quad a = b = \tfrac{1}{2}, \quad c = 1,
\]
so that the initial expressed opinions are $+1$ from $u$ and $-1$ from $v$: equal in magnitude, opposite in sign. Let $y_u$ denote the free component of $x_u$. The discrepancy is $d_e = c x_v - a - b y_u$, and with $\alpha = \beta = 1$ the dynamics are
\[
\frac{dy_u}{dt} = b d_e, \quad \frac{dx_v}{dt} = -c d_e, \quad \frac{da}{dt} = d_e, \quad \frac{db}{dt} = y_u d_e, \quad \frac{dc}{dt} = -x_v d_e,
\]
where variables are held constant according to each scenario.

\textbf{Scenario 2 (Structural Stubbornness).} Both restriction maps are frozen at $a = b = \tfrac{1}{2}$, $c = 1$. The discrepancy $d_e = x_v - \tfrac{1}{2} - \tfrac{1}{2} y_u$ drives the two-dimensional linear system $\frac{dy_u}{dt} = \tfrac{1}{2} d_e$, $\frac{dx_v}{dt} = -d_e$. The equilibrium condition $d_e = 0$ defines the line $x_v = \tfrac{1}{2}(1 + y_u)$ of global sections. Since the dynamics are gradient descent on $\tfrac{1}{2}d_e^2$, the equilibrium is the orthogonal projection of $(1, -1)$ onto this line. Minimizing $(y_u - 1)^2 + (\tfrac{1}{2}(1 + y_u) + 1)^2$ yields
\[
y_u^* = \tfrac{1}{5}, \qquad x_v^* = \tfrac{3}{5}, \qquad \text{expressed value} = \tfrac{3}{5}.
\]

\textbf{Scenario 3 (Accommodation).} The map from $u$ is frozen ($a = b = \tfrac{1}{2}$) while $c$ evolves. The key observation is that quotients of rates yield conserved quantities. From $\frac{dx_v}{dt}/\frac{dc}{dt} = c/x_v$, we obtain $x_v \, dx_v = c \, dc$, hence $x_v^2 - c^2 = 0$, giving $x_v = -c$ (taking the branch consistent with initial signs). From $\frac{dy_u}{dt}/\frac{dc}{dt} = -b/(2x_v) = 1/(2c)$, we get $dy_u = dc/(2c)$, hence $y_u - \tfrac{1}{2}\ln c = 1$. At equilibrium $d_e = 0$: $c x_v = \tfrac{1}{2}(1 + y_u)$, so $-c^2 = \tfrac{1}{2}(2 + \tfrac{1}{2}\ln c)$. Rearranging gives the transcendental equation $\ln c = -4c^2 - 4$, whose unique positive solution is $c^* \approx 0.02$. Then
\[
x_v^* = -c^* \approx -0.02, \qquad y_u^* = 1 + \tfrac{1}{2}\ln c^* \approx -1.00, \qquad \text{expressed value} \approx 0.
\]

\textbf{Scenario 4 (Outreach).} The map from $v$ is frozen ($c = 1$) while $a$ and $b$ evolve. From $\frac{dx_v}{dt} + \frac{da}{dt} = 0$, we get $x_v + a = -\tfrac{1}{2}$. From $\frac{dy_u}{dt}/\frac{db}{dt} = b/y_u$, we get $y_u^2 - b^2 = \tfrac{3}{4}$. A third conserved quantity arises from $da/dy_u = 1/b$. Initially $b = \tfrac{1}{2} > 0$, so $b = +\sqrt{y_u^2 - \tfrac{3}{4}}$ and integrating gives $a - \ln(y_u + \sqrt{y_u^2 - \tfrac{3}{4}}) = \tfrac{1}{2} - \ln\tfrac{3}{2}$. As the trajectory evolves, $y_u$ decreases toward $\sqrt{3}/2$ where $b = 0$, then $b$ becomes negative. For $b < 0$, the sign flips: $b = -\sqrt{y_u^2 - \tfrac{3}{4}}$ and continuity at the crossing gives $a + \ln(y_u + \sqrt{y_u^2 - \tfrac{3}{4}}) = \tfrac{1}{2} - \ln 2$. At equilibrium $d_e = 0$: $x_v = a + b y_u$, combined with $x_v = -\tfrac{1}{2} - a$ yields $b y_u = -\tfrac{1}{2} - 2a$. Solving numerically on the $b < 0$ branch:
\[
y_u^* \approx 0.88, \quad b^* \approx -0.13, \quad a^* \approx -0.19, \quad x_v^* \approx -0.31, \quad \text{expressed value} \approx -0.31.
\]

\textbf{Scenario 1 (Universal Adaptation).} All of $a$, $b$, $c$ evolve. As in Scenario~3, $x_v^2 - c^2 = 0$ gives $x_v = -c$. As in Scenario~4, $y_u^2 - b^2 = \tfrac{3}{4}$. From $\frac{dc}{dt}/\frac{da}{dt} = -x_v = c$, we get $dc/c = da$, hence $c = e^{a - \frac{1}{2}}$. The trajectory again crosses from $b > 0$ to $b < 0$. At equilibrium $d_e = 0$: $-c^2 = a + b y_u$, so $-e^{2a-1} = a + b y_u$. With $b = -\sqrt{y_u^2 - \tfrac{3}{4}}$ and the conserved quantities, numerical solution yields:
\[
y_u^* \approx 0.87, \quad b^* \approx -0.10, \quad a^* \approx -0.17, \quad c^* \approx 0.51, \quad x_v^* \approx -0.51, \quad \text{expressed} \approx -0.26.
\]


\section{Notation Reference}\label{app:notation}

For convenience, we collect the principal notation used throughout the paper.

\begin{table}[t]
\centering
\caption{Summary of notation.}
\label{tab:notation}
\begin{tabular}{@{}cl@{}}
\toprule
Symbol & Description \\
\midrule
\multicolumn{2}{l}{\textit{Graph and Sheaf Structure}} \\[2pt]
$G = (V, E)$ & Network graph with vertices $V$ and edges $E$ \\
$\mathcal{F},\,\mathcal{F}(v),\,\mathcal{F}(e),\,\mathcal{F}_{v\unlhd e}$ 
& Discourse sheaf with vertex/edge stalks and restriction maps \\[4pt]
\midrule
\multicolumn{2}{l}{\textit{Cochain Spaces and Operators}} \\[2pt]
$C^0(G;\mathcal{F}),\,C^1(G;\mathcal{F})$ & Spaces of 0- and 1-cochains \\
$\delta$ & Coboundary operator $C^0\to C^1$ \\
$L_{\mathcal{F}}=\delta^T\delta$ & Sheaf Laplacian \\
$H^0(G;\mathcal{F}),\,H^0(G,U;\mathcal{F})$ 
& (Relative) zeroth cohomology / global sections \\[4pt]
\midrule
\multicolumn{2}{l}{\textit{Stubborn Opinions}} \\[2pt]
$U \subseteq V$ & Set of stubborn vertices \\
$F = V \setminus U$ & Set of free vertices \\
$S_v$, $T_v$ & Stubborn and free subspaces at vertex $v$ \\
$\mathcal{Q}$ & Sheaf of free opinions \\
$\iota_S,\,\iota_{\mathcal{Q}},\,P_S,\,P_{\mathcal{Q}}$
& Inclusions and projections for stubborn/free subspaces \\
$\tilde u,\,\tilde y$ & Embedded stubborn and free opinions \\
$L_{SS}, L_{SQ}, L_{QS}, L_{QQ}$ & Block decomposition of $L_{\mathcal{F}}$ \\
$L_{\mathcal{Q}}$ & Laplacian of $\mathcal{Q}$ (equals $L_{QQ}$) \\[4pt]
\midrule
\multicolumn{2}{l}{\textit{Stubborn Expressions}} \\[2pt]
$\mathcal{I}$ & Set of adapting incidences \\
$V_{\mathrm{maps}}$ & Space of adapting restriction maps \\
$A$, $c$ & Linear operator and constant encoding discrepancy \\
$\mathcal{H}^x$ & Sheaf of structures \\
$\mathcal{H}^x_{\mathcal{I}}$ & Sheaf of free structures \\
$\mathcal{M}$ & Constraint subspace for structure velocities \\
$\Pi_{\mathcal{M}}$ & Orthogonal projection onto $\mathcal{M}$ \\[4pt]
\midrule
\multicolumn{2}{l}{\textit{Edge Classification}} \\[2pt]
$E_{FF}$, $E_{UU}$, $E_{UF}$ & Edges by endpoint type (both free, both stubborn, mixed) \\
$E_S$ & Type S edges (symmetric adaptation) \\
$E_A$ & Type A edges (asymmetric adaptation) \\[4pt]
\midrule
\multicolumn{2}{l}{\textit{Energy and Dynamics}} \\[2pt]
$\Psi(y, \delta)$ & Total disagreement energy, $\frac{1}{2}\|\delta x\|^2$ \\
$\mathcal{L}_{\lambda,\mu}$ & Regularized energy functional \\
$Q_{vv}(t)$ & Conserved quantity matrix \\[4pt]
\midrule
\multicolumn{2}{l}{\textit{Parameters}} \\[2pt]
$\alpha$, $\beta$ & Opinion and structure update rates \\
$\lambda$, $\mu$ & Regularization parameters \\
$\lambda_{\mathrm{eff}}$, $\mu_{\mathrm{eff}}$ & Effective spectral gaps (trajectory-dependent dissipation ratios) \\
$\rho_-$, $\rho_+$ & Regime threshold ratios \\
$B_x$, $B_\delta$ & Bounds on opinion and restriction map norms \\
$\varepsilon$ & Displacement tolerance \\
\bottomrule
\end{tabular}
\end{table}

\end{document}